\documentclass{amsart}
\usepackage{stmaryrd}
\usepackage{xr}
\usepackage{amscd, amssymb, calc, ifthen, mathrsfs}
\usepackage[matrix,arrow,curve,cmtip]{xy}
\usepackage{amsxtra}
\usepackage{eucal}
\usepackage{xcolor}
\usepackage{color}
\usepackage[normalem]{ulem}
\usepackage{amsxtra}
\usepackage{mathabx}
\usepackage{tikz-cd}
\usepackage{hyperref,cite}

\let \: \colon

\newcommand{\mathbold}{\mathbb}
\newcommand{\bF}{{\mathbold F}}

\newcommand{\bQ}{{\mathbold Q}}

\newcommand{\bZ}{{\mathbold Z}}
\newcommand{\bC}{{\mathbold C}}
\newcommand{\bR}{{\mathbold R}}
\newcommand{\bB}{\mathbold{B}}

\newcommand{\bN}{{\mathbold N}}
\newcommand{\Rpos}{\bR_{+}}
\newcommand{\Qpos}{\bQ_{+}}
\newcommand{\op}{\mathrm{op}}

\DeclareMathOperator{\Spec}{Spec}

\newcommand{\GL}{{\mathrm{GL}}}

\newcommand{\Gm}{{\mathbold G}_{\mathrm{m}}}

\newcommand{\vbl}{-}
\newcommand{\tn}{\otimes}           

\newcommand{\longmap}{{\,\longrightarrow\,}}
\newcommand{\longlabelmap}[1]{{\,\buildrel #1\over\longrightarrow\,}}
\def\longisomap{{\,\buildrel \sim\over\longrightarrow\,}} 

\DeclareMathOperator{\colim}{\mathrm{colim}}

\newcommand{\Gal}{{\mathrm{Gal}}}
\newcommand{\sO}{{\mathcal{O}}}

\DeclareMathOperator{\tr}{\mathrm{tr}}

\newcommand{\CMon}{\mathsf{CMon}}
\newcommand{\Mod}{\mathsf{Mod}}

\newcommand{\CAlg}[1]{\mathsf{CAlg}(#1)}

\newcommand{\Hom}{\mathrm{Hom}}

\newcommand{\Pic}{\mathsf{Pic}}
\newcommand{\Cl}{\mathrm{Cl}}

\newcommand{\Aff}{\mathsf{Aff}}

\newcommand{\mpar}[1]{}

\newcommand{\End}{{\mathrm{End}}}

\newcommand{\id}{{\mathrm{id}}}

\def\isomap{{\buildrel \sim\over\rightarrow}}

\newcommand{\Ded}{R}
\newcommand{\Dedplus}{R_+}

\newcommand{\rightlabelxyarrows}[2]{{\ar@<0.7ex>^-{#1}[r]\ar@<-0.7ex>_-{#2}[r]}}

\newcommand{\displaylabelfork}[6]{{	\entrymodifiers={+!!<0pt,\fontdimen22\textfont2>}
	\def\objectstyle{\displaystyle}
\xymatrix{{#1} \ar^-{#2}[r] & {#3} \ar@<0.7ex>^-{#4}[r]\ar@<-0.7ex>_-{#5}[r] & {#6}}}}

\newcommand{\predisplaylabelfork}[6]{{{#1} \ar^-{#2}[r] & {#3} \ar@<0.7ex>^-{#4}[r]\ar@<-0.7ex>_-{#5}[r] & {#6}}}

\newtheoremstyle{mythm}{}{}%
  {\itshape}
  {}
  {\bfseries}
  {}
  { }
  {\thmnumber{#2.\hspace{1.5mm}}\thmname{#1}\thmnote{ {\mdseries(#3)}}.}

\newtheorem*{theorem*}{Theorem} 

\newtheoremstyle{intro}{}{}%
  {\itshape}
  {}
  {\bfseries}
  {}
  { }
  {\thmname{#1}\thmnumber{ #2}\thmnote{ #3}.}
\newtheoremstyle{mythmnotitalic}{}{}%
  {}
  {}
  {\bfseries}
  {}
  { }
  {\thmnumber{#2.\hspace{1.5mm}}\thmname{#1}\thmnote{ {\mdseries(#3)}}.}
\newtheoremstyle{mythmnoperiod}{}{}%
  {}
  {}
  {\bfseries}
  {}
  { }
  {\thmnumber{#2.\hspace{0.5mm}}\thmname{#1}\thmnote{{\mdseries(#3)}}}

\numberwithin{equation}{subsection}

\theoremstyle{mythmnotitalic}
\newtheorem{remark}[subsection]{Remark}
\newtheorem{example}[subsection]{Example}
\newtheorem{definition}[subsection]{Definition}

\theoremstyle{mythm}
\newtheorem{theorem}[subsection]{Theorem}
\newtheorem{proposition}[subsection]{Proposition}
\newtheorem{lemma}[subsection]{Lemma}
\newtheorem{corollary}[subsection]{Corollary}

\theoremstyle{mythmnoperiod}
\newtheorem{ssec}[subsection]{}

\theoremstyle{intro}

\newtheorem*{prop*}{proposition}

\title{Facets of module theory over semirings}

\author[J.~Borger]{James Borger}
\address{Mathematical Sciences Institute\\ Australian National University\\Canberra ACT 0200\\Australia}
\address{International Research Laboratory\\ France Australia Mathematical Sciences and Interactions}
\email{james.borger@anu.edu.au}

\author[J.~Jun]{Jaiung Jun}
\address{Department of Mathematics, State University of New York at New Paltz, NY, USA}
\email{junj@newpaltz.edu}
\thanks{The second author acknowledges the support of an AMS-Simons Research Enhancement Grant for Primarily Undergraduate Institution (PUI) Faculty during the writing of this paper.}

\AtBeginDocument{%
   \def\MR#1{}
}
\begin{document}
	
\maketitle

\begin{abstract}
We set up some basic module theory over semirings, with particular attention to what is needed in
scheme theory over semirings.
We show that while not all the usual definitions of vector bundle agree over semirings, all the usual
definitions of line bundle do agree. We also show that the narrow class group of a number field
can be recovered as a reflexive Picard group of its subsemiring of totally nonnegative algebraic integers.
\end{abstract}

\setcounter{tocdepth}{1}
\tableofcontents

\section{Introduction}
This paper is about module theory over (commutative) semirings, with particular attention on what is needed in scheme theory over semirings. So we will say something about that first.

Scheme theory over a base ring $R$ is the syntactic theory of polynomial equations with coefficients in $R$. We
can say this because there is a category-theoretic machine which when applied to polynomials equations with
coefficients in $R$, outputs scheme theory over $R$. While scheme theory was not presented this way in the
first and most commonly cited accounts~\cite{EGA-no.4}\cite{Hartshorne:algebraic-geometry}, something
close to it did appear in other developments 
\cite{Demazure:GA}\cite{EGAI-Springer-edition}\cite{Grothendieck:Buffalo-course}\cite{Jantzen:book-first-edition}
from about the same time, taking the functor of
points as the starting point. But even these works broke what we consider
strict category-theoretic form at places. To our
knowledge, no one wrote out foundations embracing this point of view fully until
T\"oen--Vaqui\'e~\cite{Toen-Vaquie:Under-Spec-Z}.

But now that we have it, it is natural to ask whether we may similarly study polynomial equations with
coefficients in any \emph{semiring}---``rings but possibly without subtraction''---most importantly the
nonnegative integers $\bN$, rationals $\Qpos$, or reals $\Rpos$. A semiring is after all the smallest algebraic
structure closed under addition and multiplication, and hence the minimal structure that supports a theory of
polynomials. The smallest base ring allowed in algebraic geometry has progressed over time down the hierarchy of
numbers, from $\bC$ to $\bQ$ and then to $\bZ$, reversing the historical enlargement of the concept of number.
Going all the way to $\bN$ would bring this process to its natural conclusion.

The question of whether such a theory exists was in fact already answered in T\"oen and Vaqui\'e's original
paper: flatness, faithfully flat descent, the fpqc topology, and open immersions all work perfectly well over
semirings. This allows one to define schemes over any semiring, as well as the associated fibered categories of
schemes and quasi-coherent sheaves. Thus the most fundamental objects of scheme theory have definitions over
semirings in which we can have complete confidence. One simply starts with the theory of polynomial equations
over $\bN$ or $\Rpos$, or one's favorite base semiring, and applies a general category-theoretic machine.

Allowing base rings like $\bZ$ in algebraic geometry rather than only base fields like $\bQ$ gave rise to the
concept of an integral model of a scheme over $\bQ$. Until then, integrality had to be handled outside the
theory in an \emph{ad hoc} way. Similarly, admitting base semirings such as $\bN$ or $\Rpos$ gives rise to the
concept of a \emph{positive model} of a scheme over $\bZ$ or $\bR$. While positive models do exist in the
literature today, they are similarly handled in an \emph{ad hoc} way. See for example the inspiring paper of
Fock--Goncharov~\cite{Fock-Goncharov:higher-Teichmueller}.

One is further encouraged in this direction by global class field theory. In the theory of ray class groups, positivity plays the role of $p$-adic integrality at the infinite place $p=\infty$. So we might write
$$
\bQ_\infty=\bR,\quad \bZ_\infty=\Rpos,
$$
the first of which is common notation but the second of which is, we believe, new. Taking this further, $\bZ_p$ has a residue field, $\bF_p$, but does $\Rpos$? The semiring $\Rpos$ is indeed a local semiring, meaning it has a unique maximal ideal. But it is also a semifield, meaning it is nonzero and every nonzero element is multiplicatively invertible, which might suggest that it has no nontrivial quotients. However unlike in ring theory, semifields can have nontrivial quotients. In fact, $\Rpos$ has a unique one---the \emph{Boolean semifield} 
$$
\bB=\{0,1\}=\bN/(1+1\sim 1).
$$ 
The quotient map $\Rpos\to\bB$ is the unique map; it is given by
$$
x\mapsto \begin{cases} 1, & x>0 \\ 0, & x=0.\end{cases}
$$
We might therefore metaphorically call $\bB$ the residue semifield $\bF_\infty$ at the infinite place.

But is all this going to far? Transporting language from the finite places to the infinite place is admittedly
not meant to be taken too seriously. But it is harmless if kept in place and is sometimes inspiring. The more
interesting question is whether $\bB$ and similar semirings should be taken seriously as objects of interest,
regardless of the language we use for them. We submit the answer is \emph{yes}. Let us compare $\bB$ to
$\bF_2$. First, the existence of the ring $\bF_2$ is merely the statement that congruence modulo $2$, an
important property in number theory, is preserved under addition and multiplication. Likewise the existence of
$\bB$ is the statement that positivity, another important property, is similarly preserved. Using the language
of $\bB$ to study positivity in the theory of polynomials is no bigger a departure from core mathematics than is
using $\bF_2$ to study congruence modulo $2$. Second, we remind the reader that every semiring admits a
presentation $\bN[x_i\mid i\in I]/(f_j\sim g_j\mid j\in J)$. So a semiring up to isomorphism is nothing more
than a system of polynomial equations over $\bN$ up to algebraic equivalence. If a system of
equations has no solutions in any ring but does so in $\bB$, this is important information, much as would be
the case if it had no solutions in $\bZ[1/2]$-algebras but did so in $\bF_2$. There are no pathological
semirings any more than there are pathological systems of polynomial equations with coefficients in $\bN$ or,
dare we say, any more than there are pathological natural numbers.

Here are some more small facts that might inspire the reader. The category of rings is nothing more than the
category of semirings $R$ admitting a map $\bZ\to R$, which is necessarily unique. So passing from rings to
semirings is more like passing from rings to $\bZ[1/p]$-algebras than passing from rings other to
generalizations, such as possibly non-commutative rings. Further, a semiring $R$ is a ring if and only if
$\bB\otimes_\bN R=\{0\}$, much as a ring is a $\bZ[1/p]$-algebra if and only if $\bF_p\otimes_\bZ R=\{0\}$.
So we might say that the complement of the Boolean point of $\Spec(\bN)$ is
$\Spec(\bZ)$, just as the complement of the $\bF_p$-point of $\Spec(\bZ)$ is $\Spec(\bZ[1/p])$.

Many basic concepts in commutative algebra bifurcate over semirings. A large part of this paper considers this
for the concepts of line bundle and vector bundle, but we have already mentioned a simpler example---the concept
of field. There are semifields, such as $\Rpos$, which are non-simple in that they have nontrivial quotients. So
points in the Zariski topology over semirings are subtler than over rings. We might however still ask what the
simple objects in the category of semirings are. This is answered definitively by Golan's theorem: a semiring is
simple if and only if it is either a field or isomorphic to $\bB$.

Continuing, the group scheme $\mu_n$ of $n$-th roots of unity and constant group scheme $\underline{\bZ/n\bZ}$
are never isomorphic over $\bQ$---unless $n=2$. But this exception disappears over $\Qpos$, because
$\mu_2(\Qpos)$ has one element while $\underline{\bZ/2\bZ}(\Qpos)$ has two. A single-sided inverse of a square
matrix over any semiring is necessarily a two-sided inverse, which is familiar over fields and even over rings.
Over semirings, it is due to~\cite{reutenauer1984inversion}. The Cayley--Hamilton theorem is true
over semirings, once suitably formulated~\cite{Rutherford:Cayley-Hamilton}. A theory of the \'etale fundamental
group exists for schemes over semirings. This is developed in Culling's thesis~\cite{culling2019etale}.

\vspace{5mm}

What about examples? Any scheme over a ring is a scheme over $\bZ$ and hence over $\bN$. So the category of
schemes over the semiring $\bN$ contains the category of schemes over $\bZ$ as a (full) subcategory. Also, any
system of polynomial equations over any semiring defines an affine scheme over it. So there are plenty of
examples. What we want however are examples with some meaning, such as moduli spaces, over
arithmetically interesting semirings like $\bN$, $\Qpos$, or $\Rpos$. And since any moduli space over these
semirings gives after base change one over $\bZ$, $\bQ$, or $\bR$, we can equivalently ask:

\vspace{3mm}
\emph{Which interesting moduli spaces over rings like $\bZ$, $\bQ$, or $\bR$ have preferred positive models?}
\vspace{3mm}

One might say this brings to the infinite place the distinguished tradition of finding integral models of moduli spaces at finite places. One might then hope to study the Boolean fiber, no doubt degenerate, and use that as a lever to handle the original moduli space. We view this as the most important direction in scheme theory over semirings.

\vspace{5mm}

This paper began with a question of this nature, namely whether the Picard stack is algebraic, say for schemes
which are flat, projective, and of finite presentation over $\bN$ or $\Rpos$. But we realized quickly that our
test question might be many questions, because it was not clear whether the classical definitions of \emph{line
bundle} remained equivalent over semirings. So we set out to set up some of the basic linear algebra, or module
theory, over general semirings, with particular emphasis on line bundles and vector bundles. This is the purpose
of this paper.

For line bundles, it is fortunately true that all the classical definitions agree:

\begin{theorem}\label{thm-intro: line bundles}
Let $R$ be a semiring. The following conditions on $R$-modules are equivalent:
\begin{enumerate}
	\item invertible,
	\item Zariski-locally free of rank $1$,
	\item fpqc-locally free of rank $1$.
\end{enumerate}
\end{theorem}

For vector bundles, there is good and bad news:

\begin{theorem}\label{thm-intro: vector bundles}
The following conditions on finitely presented $R$-modules are equivalent:
\begin{enumerate}
	\item dualizable,
	\item projective,
	\item flat.
\end{enumerate}	
Further, all fpqc-locally free modules of finite rank satisfy these properties.
But there do exist semirings $R$ with finitely presented projective modules which are not locally free in the 
Zariski or fpqc topology.
\end{theorem}

A slightly refined expression of these theorems is depicted in table~\ref{table: implications}. 
We emphasize that at the time of this writing, we do not know whether all fpqc-locally finite free modules are Zariski-locally free. We also note that it was already observed in~\cite{jun2024vector} that there exist finitely generated projective modules which are not Zariski-locally free.

\begin{table}[t]
\label{table: implications}
\begin{tabular}{c}
$$
	\begin{tikzcd}[row sep=1.5cm, column sep=1cm]
		\text{Zar.-loc.\ free} \arrow[r] &
		\text{fpqc-loc.\ free} \arrow[dashed,r,"\text{rank}<\infty",swap] 
		\arrow[bend right=20,dashed,swap]{l}{\text{?}} &	
		\text{dualizable}  \arrow[r]  \arrow[bend right=20,dashed,swap]{l}{\text{over rings}}& 
		\text{projective} \arrow[r] \arrow[bend right=20,dashed,swap]{l}{\text{f.gen}} &
		\text{flat} \arrow[bend right=30,dashed,swap]{l}{\text{f.pres}}
	\end{tikzcd}
$$\\
\\
$$
	\begin{tikzcd}[row sep=1.5cm, column sep=1cm]
		\text{Zariski-loc.\ free of rank 1} \arrow[r,<->]  &
		\text{fpqc-loc.\ free of rank 1}  \arrow[r,<->] &	
		\text{invertible}  
	\end{tikzcd}
$$	
\end{tabular}
\vspace{\abovecaptionskip}
\caption[]{Implications between different definitions of \emph{line bundle} and \emph{vector bundle}. Dashed arrows indicate that the implication does not hold in general but does under the condition given. We add that every finitely generated projective module is finitely presented.}
\end{table}

Since there was then apparently no doubt about the definition of line bundle, we considered the following question: 

\vspace{3mm}
\emph{Is the narrow class group of a number field equal to the Picard group of its subsemiring of totally nonnegative algebraic integers? }
\vspace{3mm}

One would expect this to be true if one took seriously the analogy above, that positivity is $\infty$-adic integrality. Unfortunately, it isn't true. But it becomes true if instead of using invertible modules, we use a slight weakening: reflexive modules of rank $1$. 

To explain this, fix a number field $F$ and any set $S$ of real places, for instance all of them. Let 
${\Cl_S}(F)$ 
denote the corresponding narrow class group, which is to say the quotient of the group of fractional ideals of $F$ by the subgroup of elements of the form $(\beta)$, where $\beta\in F^*$ satisfies $\sigma(\beta)\geq 0$ for all $\sigma\in S$. 

Then let $\sO_{F_+}$ denote the subsemiring of $F$ consisting of algebraic integers which are nonnegative at all places in $S$, and let $\Pic^{\mathrm{refl}}(\mathcal{O}_{F_+})$ denote the set of isomorphism classes of reflexive $\sO_{F_+}$-modules of rank $1$. Here, a module $L$ is \emph{reflexive} if the canonical map $L\to L^{\vee\vee}$ to its double dual is an isomorphism, and $L$ is \emph{of rank 1} if $F\otimes_{\sO_{F_+}}L$ is $1$-dimensional as a vector space over $F$.

\begin{theorem}\label{thm-intro: narrow class group}
There is a canonical bijection 
${\Cl_S}(F)\longisomap \Pic^{\mathrm{refl}}(\mathcal{O}_{F_+}).$ 
\end{theorem}

It follows that the reflexive class group has a description as an adelic double quotient:
\begin{equation} 
	\Pic^{\mathrm{refl}}(\mathcal{O}_{F_+})
	= F^\times\backslash\sideset{}{'}\prod_{v} F_v^\times/\prod_{v}\mathcal{O}_{F_v}^\times,	
\end{equation}
where $v$ runs over the non-complex places and at the real places
$\mathcal{O}_{F_v}$ is understood to be $\Rpos$.

A reflexive class group has also been considered in more traditional algebraic geometry,
especially for varieties $X$ with normal singularities
where it agrees with the Weil divisor class group. (See the Stacks Project~\cite{stacks-project}, tag 0EBK.)
This suggests that one could view $\Spec(\sO_{F_+})$ as being not regular at $\infty$. 

\vspace{3mm}
\emph{Contents of the sections.}
Sections~\ref{section: preliminaries}--\ref{section: Zariski topology} are an exposition of the basic theory semirings and modules over them, including the equational criterion for flatness, faithfully flat descent, and the Zariski topology. The novelty of these sections is mostly in the exposition.

Sections~\ref{section: finite presentation and generation}--\ref{section: line bundles} cover the different definitions of vector bundle and line bundle. They culminate in the proof of theorems~\ref{thm-intro: line bundles} and~\ref{thm-intro: vector bundles} and the rest of the implications in table~\ref{table: implications}.

Sections~\ref{section: reflexive modules} and~\ref{section: narrow class group} prove theorem~\ref{thm-intro: narrow class group}.

Sections~\ref{section: GL_n is not flat} and~\ref{section: GL_n near the Boolean fiber} form an appendix on the failure of flatness of the affine group scheme $\GL_n$ over semirings like $\bN$ and $\Rpos$. We hope they reveal something about complications of vector bundles of higher rank over semirings. The main consequence of these sections is the following:

\begin{theorem}
\label{thm-intro: GL_n is not flat}
If $n\geq 2$, then the group scheme $\GL_n$ is not flat over $\bN$, $\Qpos$, or $\Rpos$ (or more generally any nonzero additively cancellative semiring which is negative free). It is however flat over the Boolean semifield $\bB$, where it agrees with the torus normalizer ${\mathbold G}_{\mathrm{m}}^n\rtimes \underline{S_n}$.
\end{theorem}

\emph{Previous work.}
We placed ourselves above in a certain tradition~\cite{Toen-Vaquie:Under-Spec-Z}\cite{borger2016witt} which takes invariance under base change as the most basic requirement. But there is much other work on algebraic geometry and linear algebra over semirings of a more absolute flavor.

As for vector bundles over semirings, we mention \cite{flores2014picard}\cite{pirashvili2015cohomology}\cite{jun2019picard}\cite{jun2024vector} \cite{gross2022principal}.

For algebraic geometry over $\bB$-algebras, there is of course the Tropical school. The literature there is
vast; so we will just cite the book of Maclagan and Sturmfels~\cite{Maclagan-Sturmfels:book}, which has an
extensive bibliography. That subject is however not practiced in a scheme-theoretic style, with the notable
exceptions
of~\cite{Giansiracusa-Giansiracusa:equations-of-tropical-varieties}\cite{Maclagan-Rincon:tropical-ideals}
\cite{Maclagan-Rincon:tropical-schemes}. It would be interesting to investigate whether there is a
scheme-theoretic approach to the topics of interest to that school in general. Finally we mention the work of
Connes--Consani, which also lives over $\bB$ but which is focused more on big applications than on developing
tight general theory. We refer the reader to~\cite{Connes-Consani:Riemann-Roch-strategy} and the references therein.

\section{Preliminaries}
\label{section: preliminaries}

This section gives an abbreviated account of some basic definitions and category-theoretic properties of 
commutative semirings and modules over them. For a fuller development, see sections 1 and 2 
of~\cite{borger2016witt}. Golan's books,~\cite{Golan:book1} and the updated~\cite{Golan:book2},
are also good references for many basics, but the reader should
beware that some of his definitions disagree with ours. 

Let $\CMon$ denote the category of commutative monoids, usually written additively. 
The category $\CMon$ has all limits and colimits, 
say because commutative monoids are described by an algebraic theory. 
(For this and other facts about algebraic theories, 
we will follow Borceux~\cite{Borceux:Handbook.v2}, chapter 3.)
It also has a theory of generators and relations, as follows. 
Given an object $M\in\CMon$,
a family of elements of $M$, written $m\:I\to M$ or $(m_i)_{i\in I}$, is a generating family if the image
of $m\:I\to M$ generates $M$, or equivalently if the induced map $F(I)\to M$ is surjective, where $F(I)=\bN^{(I)}=\bigoplus_I\bN$
is the free commutative monoid on the set $I$. We say $M$ is finitely generated if it has a finite generating
family. This is equivalent to requiring the functor $\Hom(M,\vbl)$ to preserve filtered colimits of injections.

Given an object $F\in\CMon$ (often free), a relation on $F$ is a pair $(f,g)\in F^2$. 
We will often use $f\sim g$ as an alternative notation for the pair $(f,g)$. Given a family of relations $J\to F^2$, 
or $(f_j\sim g_j)_{j\in J}$, the congruence it generates is the smallest subobject of $F^2$ which contains 
the image of $J\to F^2$ and is also an equivalence relation on $F$.
It is notated  $( f_j\sim g_j \mid j\in J )$.
Then the set $F/( f_j\sim g_j \mid j\in J)$ of equivalence classes inherits a unique commutative monoid structure from $F$. It is the
universal recipient of a morphism from $F$ in which all the relations $f_j\sim g_j$ are satisfied. A presentation
of a commutative monoid $M$ is then a set $I$, a family of relations
$(f_j\sim g_j)_{j\in J}$ on the free commutative monoid $F(I)=\bN^{(I)}$, and an isomorphism 
	$$
	\bN^{(I)}/( f_j\sim g_j \mid j\in J) \longisomap M.
	$$
One says $M$ is finitely presented if it has a presentation in which $I$ and $J$ are finite. This is equivalent
to requiring the functor $\Hom(M,\vbl)$ to preserve all filtered colimits. 

The category $\CMon$ is enriched over itself: for $M,N\in\CMon$,
the set $\Hom(M,N)$ has the structure of a commutative monoid with $+$ being pointwise addition
and $0$ being the constant zero morphism. Then a homomorphism $N\to \Hom(M,P)$
is equivalent to a function $f\:M\times N\to P$ which is bilinear, in the sense that $f(m,n)$ is a morphism in $\CMon$ in each variable if the other is held fixed.

The functor $\Hom(M,\vbl)$ has a left adjoint $M\otimes\vbl$, where $M\otimes N$ denotes the commutative
monoid generated by symbols $m\otimes n$ ($m\in M, n\in N$) modulo the relations making $m\otimes n$ bilinear in
$m$ and $n$. Since a bilinear map from $M\times N$ is essentially the same as one from $N\times M$,
we have an isomorphism $M\otimes N\to N\otimes M$, which is given by $m\otimes n\mapsto n\otimes m$. This endows
$\CMon$ with a symmetric monoidal structure under $\otimes$ with unit $\bN$.

A semiring (for the moment not necessarily commutative) 
is defined to be a monoid object $R$ in $\CMon$ with respect to this
monoidal structure. Thus giving an object $R\in \CMon$ a semiring structure is equivalent to giving it
a second monoid structure $(R,\cdot,1)$ which is bilinear, which is to say the multiplication $\cdot$
distributes over $+$ on both sides and satisfies $x\cdot 0=0=0\cdot x$. 

Given a semiring $R$ and $M\in\CMon$, an $R$-module structure on $M$ is a left action of $R$, in the sense of
monoidal categories. (The term \emph{semimodule} is also used in the literature.) This is equivalent to an
action of the monoid $(R,\cdot,1)$ on $M$ such that $rm$ is bilinear in $r\in R$ and $m\in M$. It is also
equivalent to a homomorphism $R\to\End(M)$ of semirings, where the semiring structure on $\End(M)$ is given by
composition. The category of $R$-modules is denoted $\Mod(R)$. Observe that if
the semiring $R$ is actually a ring, then an $R$-module in our sense is equivalent to an $R$-module in the usual
sense.

Also observe that an $\bN$-module is nothing more than a commutative monoid, 
just as a $\bZ$-module is nothing more than an abelian group. 
The category of abelian groups forms a full subcategory of $\CMon$. 
It is both reflective, under the group-completion functor $M\mapsto \bZ\otimes_{\bN}M$, 
and coreflective, under the invertible-elements functor $M\mapsto \Hom_{\bN}(\bZ,M)$. 

From now on, the term semiring will mean commutative semiring. 

Tensor products in $\Mod(R)$ are defined similarly to tensor products in $\CMon$. 
For $M,N\in\Mod(R)$, the commutative monoid $\Hom_R(M,N)$ of $R$-module homomorphisms has an $R$-module 
structure given by pointwise scalar multiplication. 
Then the functor $\Hom_R(M,\vbl)$ has a left adjoint $M\otimes_R\vbl$, where $M\otimes_R N$ denotes the quotient 
of $M\otimes N$ by the relations $m\otimes rn = mr\otimes n$, for all $r\in R, m\in M,n\in N$. 
This endows $\Mod(R)$ with a symmetric monoidal structure, with unit object $R$.

A (commutative) $R$-algebra is then a commutative monoid object in this monoidal category. Equivalently, it
is a semiring $S$ together with a semiring homomorphism $R\to S$. The category of $R$-algebras is denoted $\CAlg{R}$. Observe that an $\bN$-algebra is nothing more than semiring, and a $\bZ$-algebra is nothing more than a ring. Also if $R$ is a ring, then an $R$-algebra in our sense is equivalent to an $R$-algebra in the 
usual sense. The category of rings is a full subcategory of that of semirings.

The categories $\Mod(R)$ and $\CAlg{R}$ have all limits and colimits. 
They are given by the same kinds of constructions as when $R$ is a ring, as follows.
Limits and filtered colimits agree with those
in the category of sets. Coproducts in $\Mod(R)$ are given by direct sums, and coproducts in $\CAlg{R}$ are 
given by tensor products, just as when $R$ is a ring. In particular, the free $R$-module on $I$ is 
$R^{(I)}=\bigoplus_I R$ and the free $R$-algebra on $I$ is the polynomial algebra $R[x_i\mid  i\in I]$.
Limits and colimits of abelian groups and rings agree with those computed in the larger categories $\Mod(\bN)$ and $\CAlg{\bN}$.

The theory of generators and relations works the same in these categories as in $\CMon$, subject to the evident changes. For instance, a presentation of an $R$-algebra $A$ consists of a set $I$, a family of
relations $(f_j\sim g_j)_{j\in J}$, where each $f_j$ and $g_j$ lie in the polynomial algebra  $R[x_i\mid  i\in I]$, and an isomorphism
	$$
	R[x_i\mid i \in I]/( f_j\sim g_j\mid j\in J ) \longisomap A,
	$$
where now $( f_j\sim g_j\mid j\in J )$ denotes the congruence, now in the category $\CAlg{R}$, generated by the
relations $f_j\sim g_j$. This is by definition the minimal equivalence relation on $R[x_i\mid i \in I]$ containing
the pairs $(f_j,g_j)$ which is also a sub-$R$-algebra of 
$R[x_i\mid i \in I]^2$, or equivalently a subsemiring.

We will use the following notation:
\begin{align*}
	\Rpos	&= \{x\in\bR\mid x\geq 0\} \\
	R_+ 	&= R\cap \Rpos^n, \text{ for } R\subseteq \bR^n \\
	\bN		&= \bZ_+ = \{0,1,2,\dots\} \\
	\bB		&= \bN/(1+1\sim 1) = \{0,1\}\\
	R^{\times}		&= \{x\in R\mid \exists y\in R.\; xy=1\}, \text{ for any semiring } R
\end{align*}
An element of $R^{\times}$ is called a \emph{unit} of $R$.

An \emph{ideal} of a semiring $R$ is by definition a sub-$R$-module of $R$. 
A semiring $R$ is \emph{local} if whenever a finite (possibly empty) sum $\sum_i x_i$ is a unit one of the $x_i$ is a unit, 
or equivalently that $R$ has a unique maximal (proper) ideal.
It is a \emph{semifield} if it is nonzero and every nonzero element is a unit.

A commutative monoid $M$ is \emph{additively cancellative} if the implication 
$$x+z=y+z \quad\Rightarrow\quad x=y$$ 
holds in
$R$. It is \emph{negative free} (also called \emph{zero-sum free}) if 
$$x+y=0\quad \Rightarrow\quad x=y=0$$ 
holds.

\begin{proposition}\label{pro:component cover of disjoint union}
	Consider finitely many semirings $R_1,\dots,R_n$, and let $R=R_1\times\cdots \times R_n$.
	View each $R_i$ as an $R$-algebra via the projection $R\to R_i$ onto the $i$-th factor.
	Then we have $R_i=R[1/e_i]$, where $e_i$ is the $i$-th primitive idempotent $(0,\dots,0,1,0,\dots,0)\in R$,
	and	the induced functor
		$$
		\Mod(R) \longisomap \Mod(R_1)\times\cdots\times \Mod(R_n)
		$$
	is an equivalence.
\end{proposition}
\begin{proof}
This is straightforward. For instance, see lemmas 2.61 and 2.62 of Culling's thesis~\cite{culling2019etale}.
\end{proof}

\begin{theorem}[Golan]
	\label{thm:Golan}
	Let $A$ be a nonzero $\bB$-algebra. Then there exists a morphism $A\to\bB$.
\end{theorem}
\mpar{Prove this here?}
\begin{proof}
	\cite{Golan:book2}, p.\ 99, (8.11)
\end{proof}

\begin{definition}\label{definition: flat}
Let $R$ be an $\mathbb{N}$-algebra. An $R$-module $M$ is said to be \emph{flat} if the functor $M\otimes_R-:\Mod(R) \to \Mod(R)$  preserves all finite limits. 
\end{definition}

We note that this is equivalent to preserving equalizers of pairs. Indeed, preservation of all finite limits is equivalent to the preservation of all finite products and equalizers of pairs. But finite products are always preserved, since they are also finite sums.

\begin{proposition}\label{pro:flat-tensor-preserve-injectivity}
	Let $M\to N$ be an injective homomorphism of $R$-modules, and let $P$ be a flat $R$-module. Then the induced
	map $P\otimes_R M \to P\otimes_R N$ is injective.
\end{proposition}
\begin{proof}
	The map $M\to N$ is injective if and only if the diagonal map $M\to M\times_N M$ is an isomorphism. 
	This property is preserved by tensoring with a flat module, since the right-hand side is a finite limit.
\end{proof}

At the time of this writing, we do not know whether a module $M$ is flat if and only if $M\otimes_R\vbl$ preserves injectivity.

\begin{proposition}\label{pro:neg-free-and-canc-flat-local}
	If a semiring $R$ is additively cancellative (resp.\ negative free), then so is any flat $R$-module.
\end{proposition}
\begin{proof}
	Both properties can be expressed in terms of finite limits. See~\cite{borger2016witt}, (2.9) and (2.12).
\end{proof}

\section{Equational criterion for flatness}
\label{section: equational crierion for flatness}

This section is only used in the proof of~(\ref{proposition: f.pres flats are projective}).
But we believe it has some independent interest. So we include a developed form of it here.
It was first shown by Katsov~\cite{Katsov:flat-semimodules} that the equational criterion holds over semirings.

\subsection{Relations in modules}
Let $R$ be an $\bN$-algebra. Fix $m\geq 0$ and a pair $(r,s)\in R^m\times R^m$.
Suppose some elements $x_1,\dots,x_m$ in an $R$-module $M$ satisfy
\begin{equation}
	\label{eq:abstract-relation}
	\sum_{i=1}^m r_i x_i = \sum_{i=1}^m s_i x_i.
\end{equation}
Such a relation is said to be of \emph{type} $(r,s)$. We will call it \emph{rectifiable} if there exist $n\geq 0$ and $y_1,\dots,y_n\in M$ such that each
$x_i$ can be written
\begin{equation}
	\label{eq:a-matrix}
	x_i=\sum_{j=1}^n a_{ij}y_j
\end{equation}
where $a_{ij}\in R$ such that for all $j$, we have
\begin{equation}
	\label{eq:coord-rel}
	\sum_{i=1}^m r_i a_{ij} = \sum_{i=1}^m s_i a_{ij}.
\end{equation}
The reason for the name is that one can interpret rectifiability as
the existence of a coordinate system (generally not free) with respect to which the relation comes from
a relation on coordinates.

\begin{proposition}\label{pro:rect-interp}
The relation \eqref{eq:abstract-relation} is rectifiable if and only if the map $R^m\to M$ sending $e_i\mapsto x_i$ factors through a free module $R^n$ such that the image of $(r,s)$ under the induced map $R^m\times_M R^m\to R^n\times_M R^n$ is a redundant relation, that is, it is contained in the diagonal.
\end{proposition}
\begin{proof}
The requirement that it factors through $R^n$ is equivalent to the existence of $y_j$ and $a_{ij}$ satisfying \eqref{eq:a-matrix}. Given this, \eqref{eq:coord-rel} is equivalent to the image of $(r,s)$ lying in the diagonal.
\end{proof}

\begin{proposition}\label{pro:rect-rel}
If the functor $M\tn_R\vbl$ takes the equalizer diagram
\begin{equation}	\label{eq:E_rs}
\begin{tikzcd}[column sep=1.2cm]
E_{r,s} \arrow[r] 
	& R^m \arrow[r,shift left,"e_i\mapsto r_ie_i"] \arrow[r,swap,shift right,"e_i\mapsto s_ie_i"] 
	& R
\end{tikzcd}
\end{equation}
to an equalizer diagram, then all relations in $M$ of type $(r,s)$ are rectifiable.
\end{proposition}
\begin{proof}
We have an equalizer diagram
\begin{equation}	
	\begin{tikzcd}[column sep=1.2cm]
		M\otimes_RE_{r,s} \arrow[r] & M^m \arrow[r,shift left,"e_i\mapsto r_ie_i"] \arrow[r,swap,shift right,"e_i\mapsto s_ie_i"] & M.
	\end{tikzcd}
\end{equation}	
Therefore if $x_1,\dots,x_m\in M$ satisfy \eqref{eq:abstract-relation}, then the vector 
	$(x_1,\dots,x_m)$ is in the equalizer displayed. So there are elements $y_1,\dots, y_n\in M$ and 
	$a_1,\dots,a_n\in E_{r,s}$ such that $\sum_j y_j a_j = (x_1,\dots,x_m)$. This is exactly what it 
	means for the relation to be rectifiable.
\end{proof}

\subsection{Filtered colimits}\label{subsection: filtered colimits}
For any $R$-module $M$, let $FP_M$ denote the full subcategory of modules mapping to $M$ spanned
by the finitely presented modules. Then $FP_M$ is a filtered category, and its colimit in the
category of $R$-modules is $M$. These are general facts from category theory.

\begin{theorem}[Lazard's theorem for semirings]\label{thm: equivalent conditions for flatness}
	The following are equivalent.
	\begin{enumerate}
		\item 
Any map $N\to M$ from a finitely presented $R$-module $N$ factors through a finite free $R$-module,
		\item 
$FP_M$ has a cofinal family of finite free $R$-modules,
		\item 
$M$ is isomorphic to a filtered colimit of free $R$-modules,
		\item 
$M$ is flat,
		\item 
all relations in $M$ are rectifiable.
	\end{enumerate}
\end{theorem}
\begin{proof}
$(1) \Rightarrow (2) \Rightarrow (3) \Rightarrow (4)$: Clear.
	
$(4)\Rightarrow(5)$: This follows from~(\ref{pro:rect-rel}).
	
$(5)\Rightarrow (1)$: We will show there is a free object $P\in FP_M$ and a map $N\to P$ in $FP_M$. Let $R^p\rightrightarrows R^m\to N$ be a presentation. If $p=0$, we can take $P=N$. If $p\geq 1$, it is enough by induction to show we can find such a $P\in FP_M$ which may not be free but does admit a finite presentation with $p-1$ relations.
	
Let $(r,s)\in R^m\times_N R^m$ be the $p$-th relation in our presentation of $N$. The image of $(r,s)$ in $R^m\times_M R^m$, viewed as a relation between elements of $M$, is rectifiable by hypothesis. Therefore by~(\ref{pro:rect-interp}), there exists a commutative diagram of solid arrows
	$$
	\xymatrix{
		R^p  \ar@<0.5ex>[r]\ar@<-0.5ex>[r] \ar@{=}[d]
		& R^m \ar[d] \ar@{->>}[r] & N \ar@{-->}[d]\ar[ddr] \\
		R^p  \ar@<0.5ex>[r]\ar@<-0.5ex>[r]& R^n \ar@{-->>}[r]\ar[drr] & P \ar@{-->}[dr] \\
		& & & M
	}
	$$
	such that the relation $(r,s)\in R^m\times_M R^m$ maps to the diagonal in $R^n\times_M R^n$.
	Now let $P$ denote the coequalizer of the
	two maps $R^p\rightrightarrows R^n$. Then the map $N\to M$ factors (uniquely) through $P$ in
	such a way that the diagram above including the dashed arrows commutes. Finally since the
	$p$-th relation in the presentation of $P$ is redundant, it has a presentation with
	$p-1$ relations.
\end{proof}

\begin{corollary}[Equational criterion for flatness]\label{cor:equational-criterion}
An $R$-module $M$ is flat if and only if whenever we have
	$$rx=sx,$$ 
with matrices $r,s\in R^{1\times m}$ and $x\in M^{m}$, 	
there exist matrices $a\in R^{m\times n}$, and $y\in M^{n}$, for some $n\geq 0$, such that 
	$$
	x=ay, \quad ra=sa.
	$$
\end{corollary}

\begin{remark}
	For a module to be flat, it is not enough for all relations in a particular presentation to be
	rectifiable. For instance, the $\bZ$-module $(\bZ x\oplus\bZ y)/(x,x+2y)$ is not flat, but
	the relations $x=0$ and $x+2y=0$ are both rectifiable. In other words, rectifiable relations
	can imply non-rectifiable ones. In this case, the relation $2(x+y)=0$ is non-rectifiable.
\end{remark}

\begin{example} \label{eg:square-cone}
	The square cone is not flat over $\Rpos$.
	More precisely, let $M$ be the sub-$\Rpos$-module of $\bR^3$
	generated by the vectors 
	\begin{align*}
		x_1 &= (+1,0,1) \\
		x_2 &= (0,+1,1) \\
		x_3 &= (0,-1,1) \\
		x_4 &= (-1,0,1).
	\end{align*}
	Then $M$ is not a flat $\Rpos$-module. 
	(This is generalized in~(\ref{thm:fg-projective-Rplus-modules-are-free2}) below.)
	Indeed, suppose it were.
	Then by the equational criterion applied to the relation
		$$
		x_1+x_4=x_2+x_3,
		$$ 
	there would be elements $y_j\in M$ and $a_{ij}\in \Rpos$ such that for each $i$
	\begin{equation} \label{eq: equ criterion example 2}
		x_i = \sum_j a_{ij} y_j		
	\end{equation}
	and for each $j$
	\begin{equation}\label{eq: equ criterion example}
		a_{1j}+a_{4j} = a_{2j}+a_{3j}.		
	\end{equation}
	Now, first observe that we may assume all the $y_j$ are nonzero, for otherwise we could remove them from the 
	list and the equations above would still hold. Then 
	since the elements $x_i$ are on the boundary of $M$, they are indecomposable up to multiplication by
	scalars.
	So the relation (\ref{eq: equ criterion example 2}) implies that for each $i$ and $j$, either $a_{ij}=0$
	or $y_j=\lambda_{ij}x_i$ for some $\lambda_{ij}\in\bR_{>0}$.
	Then, since none of the $x_i$ is a multiple of any of the others, for any given $j$, at most one of the 
	the $a_{ij}$ is nonzero. 
	
	But the equation (\ref{eq: equ criterion example}) then implies that all $a_{ij}$ are zero.
	It follows that each $x_i$ is zero, which is a contradiction.
\end{example}

\section{Faithfully flat descent}
\label{section: faithfully flat descent}
In this section, we set up some formal descent theory over semirings. 
It is essentially the same as the usual theory over rings.
Any reader so inclined is encouraged to skip ahead to the following sections and refer back as necessary.

An $R$-module $S$ is \emph{faithful} if for any map $\varphi:M\to N$ of $R$-modules, $\varphi$ is an isomorphism whenever the induced map $\varphi_S:S\otimes_RM \to S\otimes_RN$ is an isomorphism. A morphism $R\to S$ of semirings is \emph{faithfully flat} if $S$ is flat and faithful as an $R$-module. We recall the following (see \cite{Toen-Vaquie:Under-Spec-Z}\cite{borger2016witt}):

\begin{theorem}[Faithfully flat descent]\label{thm:faithfully-flat-descent}
Let $S$ be a faithfully flat $R$-algebra. Then the base-change functor $S\otimes_R\vbl\:\Mod(R)\to\Mod(S)$ is comonadic.
\end{theorem}
\begin{proof}
This follows straight from Beck's Monadicity Theorem. (See Borceux~\cite{Borceux:Handbook.v2}, theorem 4.4.4, for example.)
Let $F$ denote the functor $S\otimes_R\vbl$.
First, $F$ has a right adjoint, namely the forgetful functor. 
Second, $F$ reflects isomorphisms by assumption, because $S$ is faithful.
Third, equalizers of all pairs exist in $\Mod(R)$ and $F$ preserves them, because $S$ is flat.
Therefore by Beck's Theorem, $F$ is comonadic.
\end{proof}

\begin{remark}
	\label{rmk:fpqc-topology}
Let $\Aff_{R}=\CAlg{R}^{\op}$. 
For any $R$-algebra $A$, let $\Spec(A)$ denote the object of $\Aff_R$ corresponding to $A$.
For any semiring $R$, we can then define the fpqc topology on $\Aff_{R}$ to be the one generated
by finite families $(\Spec(A_i)\to\Spec(A))_{i\in I}$, 
where $\prod_{i\in I} A_i$ is a faithfully flat $A$-algebra. 
(This agrees with the usual one over rings, as defined 
in~\cite[\href{https://stacks.math.columbia.edu/tag/022A}{Tag 022A}]{stacks-project} say.)
Then by faithfully flat descent, 
the categories $\Mod(A)$ as $A$ varies over $\CAlg{R}$ form a stack $\mathsf{Mod}_R$ in the fpqc topology. 
It could also called the stack of quasi-coherent sheaves and denoted $\mathsf{QCoh}_R$, which  
is perhaps more common when $R$ is a ring.

For formal reasons, many of the results in this paper can be extended to nonaffine schemes over semirings.
We hope to return to this in a future paper.
\end{remark}

\begin{definition}\label{definition: pure morphism}
	A morphism $\alpha:R \to S$ of semirings is said to be \emph{pure} if for any $R$-module $M$, 
	the following is an equalizer diagram in $\Mod(R)$:
	\begin{equation}\label{eq: equalizer pure}
		\begin{tikzcd}[row sep=large, column sep=1.5cm]
			M \arrow[r,"\alpha_M"]&
			S\otimes_RM \arrow[r, shift left,"\alpha_S \otimes_R 1_M"] \arrow[r,swap, shift right,"1_S \otimes_R \alpha_M"]
			& S \otimes_R S \otimes_R M
		\end{tikzcd}
	\end{equation}
	where $\alpha_N$ denotes, for any $R$-module $N$, the map $N\to S\otimes_R N$ given by $n\mapsto 1\otimes n$.
\end{definition}

\begin{remark}
For rings,~(\ref{definition: pure morphism}) is equivalent to $\alpha$ being faithful. See \cite{andre2022canonical}, proposition 2.2.
\end{remark}

Let $\alpha: R \to S$ be a morphism of semirings, and $M$ be an $S$-module. We will consider $S \otimes_RM$ and $M\otimes_R S$ with their evident $S\otimes_RS$-module structures. As over rings, we define the following. 

\begin{definition}
	A \emph{descent datum} on $M$ is an $S\otimes_RS$-linear isomorphism 
	\[
	\phi: M \otimes_R S \to S\otimes_R M
	\]
	such that $\phi_1\phi_3=\phi_2$ and $\lambda_M\theta_M =1_M$, where
	\begin{enumerate}
		\item 
		$\lambda_M: S \otimes_RM \to M$ is the external multiplication map $s\otimes m\mapsto sm$.
		\item 
		$\theta_M: M \to S \otimes_R M$ is given by $m \mapsto \phi(m \otimes 1_S)$ and
		$\phi_i$ is the morphism obtained by tensoring $\phi$ with $\id_S$ in the $i$-th position for $i \in \{1,2,3\}$. For example, if $\phi(m\otimes s) = \sum_j s_j \otimes m_j$, then $\phi_2(m\otimes t\otimes s) = \sum_j s_j \otimes t \otimes m_j$. 
	\end{enumerate}
\end{definition}

Let $DD(\alpha)$ be the category of descent data. Note that the $S$-module $M_S:=S\otimes_R M$ for an $R$-module $M$ is equipped with a natural descent datum. We write $C_\alpha(M)$ for the corresponding object of $DD(\alpha)$. Therefore we have the following commutative diagram:
\begin{equation}\label{eq: descent comparison diagram}
	\begin{tikzcd}
		\Mod(R) \arrow{rr}{S\otimes_R-}
		\arrow{rd}[swap]{C_\alpha}
		&~
		&\Mod(S) \\
		&DD(\alpha)\arrow{ru}[swap]{\mathcal{F}}
	\end{tikzcd}
\end{equation}
where $\mathcal{F}$ is the forgetful functor. 

\begin{definition}
The homomorphism $\alpha\:R\to S$ is said to be a \emph{descent} morphism if the functor $C_\alpha$ of (\ref{eq: descent comparison diagram}) is fully faithful and an \emph{effective descent} morphism if $C_\alpha$ is an equivalence of categories. 
\end{definition}

\begin{lemma}\label{lemma: DD(a) split equalizer}
	Let $\alpha: R \to S$ be a morphism of semirings and $(M,\phi)$ an object in $DD(\alpha)$. Then the following diagram is a split equalizer:
	\begin{equation}\label{eq: split equalizer}
		\begin{tikzcd}[row sep=large, column sep=1.5cm]
			M \arrow[r,"\theta_M"]&
			S\otimes_RM \arrow[r, shift left,"1_S \otimes \alpha_M"] \arrow[r,swap, shift right,"1_S\otimes\theta_M"]
			& S \otimes_R S \otimes_R M
		\end{tikzcd}
	\end{equation}
\end{lemma}
\begin{proof}
	We claim that \eqref{eq: split equalizer} is a split equalizer with splittings $\lambda_M:S\otimes_RM \to M$ (the external multiplication on $M$) and $\mu_S\otimes 1_M:S\otimes_R S \otimes_R M \to S\otimes_R M$, where $\mu_S$ is the multiplication of $S$. In other words, we check the following:
	\begin{enumerate}
		\item 
		$(1_S \otimes \theta_M)\theta_M = (1_S\otimes \alpha_M)\theta_M$, 
		\item 
		$\lambda_M\theta_M = 1_M$, 
		\item 
		$(\mu_S \otimes 1_M)(1_S \otimes \alpha_M) = 1_{S\otimes_R M}$, 
		\item 
		$(\mu_S \otimes 1_M)(1_S \otimes \theta_M) = \theta_M\lambda_M$.
	\end{enumerate}
	One can easily see that $(1)$ is just the cocycle condition and $(2)$ holds by definition of a descent datum on $M$. Furthermore, one can readily check $(3)$ and $(4)$. For instance, we have
	\[
	(\mu_S\otimes 1_M)(1_S\otimes \alpha_M) (s\otimes m) = (\mu_S\otimes 1_M)(s\otimes 1\otimes m)=s\otimes m. 
	\]
\end{proof}

\begin{lemma}\label{lemma: descent = pure}
	Let $\alpha:R \to S$ be a morphism of semirings. The following are equivalent:
	\begin{enumerate}
		\item
		$\alpha$ is a descent morphism.
		\item
		$\alpha$ is a pure morphism. 
	\end{enumerate}
\end{lemma}
\begin{proof}
	$(1) \Rightarrow (2)$: 
	By the tensor-hom adjunction, we have the following commutative diagram for any $R$-modules $M_1$ and $M_2$:
	\begin{equation}\label{eq: pure descent commutative diagram}
		\begin{tikzcd}[row sep=0.5cm, column sep=0.4cm]
			\Hom_R(M_1,M_2) \arrow[r] \arrow[d,"\id"]&
			\Hom_S(S\otimes_RM_1,S\otimes_RM_2) \arrow[r, shift left] \arrow[r,swap, shift right] \arrow[d,"\simeq"]
			& \Hom_S(S\otimes_RM_1,S\otimes_RS\otimes_R M_2) \arrow[d,"\simeq"]\\
			\Hom_R(M_1,M_2) \arrow[r] &
			\Hom_R(M_1,S\otimes_R M_2) \arrow[r, shift left] \arrow[r,swap, shift right] & \Hom_R(M_1,S\otimes_RS\otimes_RM_2)
		\end{tikzcd}
	\end{equation}
	Since the comparison functor $C_\alpha$ is fully faithful, 
	the top row is an equalizer diagram. Hence so is the bottom row.
	Then taking $M_1=R$, we obtain the following equalizer diagram:
	\begin{equation}\label{eq: equalizer}
		\begin{tikzcd}
			M_2 \arrow[r]&
			S\otimes_RM_2 \arrow[r, shift left] \arrow[r,swap, shift right]
			& S \otimes_R S \otimes_R M_2. 
		\end{tikzcd}
	\end{equation}
	This shows that $\alpha$ is pure. 
	
	$(2) \Rightarrow (1)$: Since $\alpha$ is pure, the bottom row of 
	\eqref{eq: pure descent commutative diagram} is an equalizer diagram for any $M_1$ and $M_2$. 
	Therefore so is the top row, and hence $C_\alpha$ is fully faithful. 
	Therefore $\alpha$ is a descent morphism. 
\end{proof}

One may apply a similar argument as in \cite[Theorem 6.3]{janelidze2004facets} to prove the following. 

\begin{proposition}
	Let $\alpha:R \to S$ be a morphism of semirings. Then, the base change functor $S\otimes_R - :\Mod(R) \to \Mod(S)$ is comonadic if and only if the (induced) base change functor $S\otimes_R-:\CAlg{R} \to \CAlg{S}$ is comonadic. 	
\end{proposition}

Let $\alpha:R \to S$ be a morphism of semirings. The base change functor $\mathcal{F}=S\otimes_R-:\Mod(R) \to \Mod(S)$ has an obvious right adjoint functor, namely the restriction (of scalars) functor $\mathcal{U}:\Mod(S) \to \Mod(R)$. Let $\eta:1_{\Mod(R)} \to \mathcal{U}\mathcal{F}$ be the unit and $\epsilon:\mathcal{F}\mathcal{U}\to 1_{\Mod(S)} $ be the counit of the adjunction. So, we have a comonad $\mathcal{C}:=\mathcal{F}\mathcal{U}$ on $\Mod(S)$ with the counit $\epsilon$ and the comultiplication $\delta=\mathcal{F}\eta_\mathcal{U}:\mathcal{C} \to \mathcal{C}^2$. A $\mathcal{C}$-coalgebra consists of a pair $(N,h:N \to \mathcal{C}(N))$ of an $S$-module $N$ and an $S$-module map $h: N \to \mathcal{C}(N)$\footnote{The $S$-module structure of $\mathcal{C}(N)=S \otimes_RN$ is obtained by $s\cdot (a\otimes n)=(sa)\otimes n$.} 
satisfying the following conditions:
\begin{enumerate}
	\item(co-associativity) 
	\begin{equation}
		\begin{tikzcd}
			N \arrow[r,"h"]&
			\mathcal{C}(N)\arrow[r, shift left,"\delta_N"] \arrow[r,swap, shift right,"\mathcal{C}(h)"]
			& \mathcal{C}^2(N),
		\end{tikzcd}
	\end{equation}	
	\item(counit) 
	\begin{equation}
		\begin{tikzcd}
			N\arrow[r,"h"]  \arrow[dr,swap,"\id_N"]&
			\mathcal{C}(N) \arrow[d,"\epsilon_N"]  \\
			& N
		\end{tikzcd}
	\end{equation}
\end{enumerate}

If $(N,h)$ is a $\mathcal{C}$-coalgebra, then one can easily see that $h:N \to \mathcal{C}(M)=S\otimes_R N$ determines a descent datum on $N$. Conversely, any descent datum $\theta:N \to S\otimes_RN$ gives rise to a $\mathcal{C}$-coalgebra $(N,\theta)$. In fact, one can check that this induces an equivalence of categories between the category $DD(\alpha)$ of descent data and the category of $\mathcal{C}$-coalgebras. We thus have the following. 

\begin{proposition}\label{proposition: eff descent = comonadic}
	A morphism $\alpha:R \to S$ of semirings is an effective descent morphism if and only if the base change functor $S\otimes_R-:\Mod(R) \to \Mod(S)$ is comonadic.  
\end{proposition}

\begin{example}\label{example: split}
	If a morphism of semirings $\alpha:R \to S$ splits, then the base change $S\otimes_R-:\Mod(R) \to \Mod(S)$ is comonadic; this directly follows from \cite[Corollary 4.2]{janelidze2004facets}. In particular, from~(\ref{proposition: eff descent = comonadic}), any split morphism of semirings is an effective descent morphisms. For instance, by Golan's theorem~(\ref{thm:Golan}), 
the inclusion $j:\mathbb{B} \to R$ is an effective descent morphism for any nonzero idempotent semiring $R$.
\end{example}

One has the following which shows that ``faithfully flat = pure + flat'' for semirings.

\begin{proposition}\label{proposition: faithfully flat descent}
Let $\alpha:R\to S$ be a flat morphism. Then $\alpha$ is faithful if and only if $\alpha$ is pure. 	
\end{proposition}
\begin{proof}
(Only if): This directly follows from~(\ref{lemma: descent = pure}) since any faithfully flat morphism is a descent morphism. 

(If): Suppose that $\varphi:M\to N$ is an $R$-module map such that $\varphi_S:M_S \to N_S$ is an isomorphism. Then, we have the following diagram:
\begin{equation}
\begin{tikzcd}
	M \arrow[r] \arrow[d,swap,"\varphi"] & S\otimes_RM \arrow[r,shift right] \arrow[r,shift left] \arrow[d,"\simeq"]& S\otimes_RS\otimes_SM\arrow[d,"\simeq"] \\
		N \arrow[r] & S\otimes_RN \arrow[r,shift right] \arrow[r,shift left] & S\otimes_RS\otimes_SN
\end{tikzcd}
\end{equation}
Since $\alpha$ is pure, the rows are equalizers. It follows that $\varphi$ is an isomorphism.
\end{proof}

\section{Zariski topology}
\label{section: Zariski topology}

\begin{proposition}\label{pro: fpqc is zariski}
Consider elements $a_1,\dots,a_m \in R$. Then  $R \to \prod_{i=1}^mR[1/a_i]$ is faithfully flat if and only
if $(a_1,\dots,a_m)$ is the unit ideal.
\end{proposition}
\begin{proof}
(If):
As in the case for rings, localization is flat.\footnote{See \cite[Section 1.17]{borger2016witt}. Also, see \cite[Example 2.94]{culling2019etale} for an interesting discussion on ``additive localization'' for semirings which is not flat.} 
So by~(\ref{proposition: faithfully flat descent}),
it is enough to prove that for any $R$-module $M$, the following is an equalizer diagram: 
\begin{equation}
\begin{tikzcd}
M \arrow[r,"f"] &	\prod_{i=1}^mM[1/a_i] \arrow[r,swap, shift right,"\pi_2"] \arrow[r,shift left,"\pi_1"] &  \prod_{i,j}M[1/a_ia_j]
\end{tikzcd}, \quad a_i,a_j \in R.
\end{equation}

We first prove that $f$ is an injection. Suppose that $t,u \in M$ have the same image. 
Then there exists $n \in \mathbb{N}$ such that $a_i^nt=a_i^nu$. 
Now because the $a_i$ generate the unit ideal, we can write $1=\sum_{i=1}^ma_ix_i$, where $x_i\in R$.
Then for any $N\geq 0$, we have
\begin{equation}
	\label{eq:hokey}
	1=\left( \sum_{i=1}^ma_ix_i\right)^{N} = \sum_{\stackrel{j_1+\cdots+j_m}{=N}} {N\choose j_1,\dots,j_m} \prod_{i=1}^m (a_ix_i)^{j_i}. 
\end{equation}
Once $N$ is sufficiently large (greater than $m(n-1)$), in each term 
$\prod_{i=1}^m(a_ix_i)^{j_i}$ one of the $j_i$ will be at least $n$, and hence we have
\[
t\prod_{i=1}^m(a_ix_i)^{j_i}=u\prod_{i=1}^m(a_ix_i)^{j_i}.
\]
Summing up according to (\ref{eq:hokey}), we have $t=u$.

Now, suppose an element $(\dots,t_i/a_i^n,\dots) \in \prod_{i=1}^mM[1/a_i]$ satisfies $t_i/a_i^n=t_j/a_j^n$ in $M[1/a_ia_j]$. Then for all $i$ and $j$, there exists an $h \in \mathbb{N}$ such that 
\[
(a_ia_j)^ha_j^nt_i = (a_ia_j)^ha_i^nt_j.
\]
Replacing each $h$ with the largest of them, we may assume $h$ is independent of $i$ and $j$. 	
Next, for any sufficiently large integer $N$ (greater than $m(n+h-1)$, say), we can write
\[
\left( \sum_{i=1}^ma_ix_i\right)^{N} = \sum_{i=1}^m a_i^{n+h}\gamma_i,\textrm{ where } \gamma_i \in R.
\]
Then we have $1=\sum_{i=1}^ma_i^{n+h}\gamma_i$. Now put $v=\sum_i a_i^h\gamma_it_i \in M$. Then we have
\[
a_j^{n+h}v=\sum_ia_j^{n+h}a_i^h\gamma_it_i = \sum_i a_j^ha_i^{n+h}\gamma_it_j = t_ja_j^h\sum_{i=1}^ma_i^{n+h}\gamma_i=t_ja_j^h. 
\]
It follows that $v$ maps to $t_j/a_j^n$ in $M[1/a_j]$. 

(Only if): Let $I$ denote the ideal generated by $a_1,\dots,a_m$. Consider the inclusion $I\to R$.
After base change to any $R[1/a_i]$, it becomes an isomorphism. It therefore becomes
an isomorphism after base change to $\prod_{i=1}^mR[1/a_i]$, which is faithfully flat by assumption.
Therefore the original inclusion $I\to R$ is an isomorphism.
\end{proof}

\begin{remark}
\label{rmk:Zariski topology on Aff}
For any semiring $R$, we can define the Zariski topology on $\Aff_{R}$ to be the one generated by finite 
families $(\Spec(A[1/a_i])\to\Spec(A))_{i\in I}$,
where the $a_i$ generate the unit ideal of $A$. It is coarser than the fpqc topology.

The Zariski topology of To\"en--Vaqui\'e~\cite{Toen-Vaquie:Under-Spec-Z} is generated by finite covering
families $(\Spec(A_i)\to\Spec(A))_{i\in I}$, where each map $A\to A_i$ is a flat epimorphism of finite 
presentation. 
It is at least as fine as our Zariski topology and is equivalent over rings.
But at the time of this writing, we do not know whether the two agree over general semirings. 
\end{remark}

\begin{remark}
\label{rmk:Jaiung's thesis}
One may also set up scheme theory over semirings using ``locally semiringed'' spaces,
as over rings with locally ringed spaces. See section 2.2 of~\cite{jun2015algebraic}. 
This is equivalent to the Zariski topology of this paper in~(\ref{rmk:Zariski topology on Aff}).
\end{remark}

\section{Finite presentation and generation}
\label{section: finite presentation and generation}

\begin{proposition}
\label{pro:direct image finiteness}
	Let $R\to S$ be a morphism of semirings.
	\begin{enumerate}
		\item If $S$ finitely generated (resp.\ presented) as an $R$-module,
			then any finitely generated (resp.\ presented) $S$-module is finitely generated (resp.\ presented) 
			as an $R$-module.
		\item If $S$ finitely generated (resp.\ presented) as an $R$-algebra,
			then any finitely generated (resp.\ presented) $S$-algebra is finitely generated (resp.\ presented) 
			as an $R$-algebra.
	\end{enumerate}
\end{proposition}
\begin{proof}
	(1): For finite generation, write such an $S$-module $M$ as a quotient of the $m$-fold coproduct
	$S^m$ for some $m\geq 0$. 
	Since $S$ is finitely generated as an $R$-module, $M$ is a quotient of a finitely generated $R$-module.
	It is therefore finitely generated as an $R$-module.
	
	Similarly, for finite presentation, write $M$ as a colimit $\colim_{i\in I} S$ over a finite diagram $I$.
	Then since $S$ is finitely presented as an $R$-module, $M$ is a finite colimit of finitely presented
	$R$-modules. It is therefore finitely presented as an $R$-module.
	
	(2): The same argument works with algebras, but with coproducts and colimits of algebras instead with of modules.
\end{proof}

\begin{proposition}\label{proposition: fpqc-locally f.pres implies f.pres}
	Let $R$ be a semiring and $R'$ be faithfully flat $R$-algebra.
	Let $M$ be an $R$-module such that $M_{R'}$ is a finitely presented (resp.\ finitely generated) $R'$-module. 
	Then $M$ is finitely presented (resp.\ finite generated) as an $R$-module. 
\end{proposition}
\begin{proof}
	For finite presentation, it is enough to show that $\Hom_R(M,\vbl)$ commutes with filtered colimits. 
	Let $(N_i)_{i \in I}$ be a filtered diagram of $R$-modules. Let $R''=R'\otimes_R R'$. 
	Consider the following diagram:
	\begin{equation}
		\begin{tikzcd}
			\colim_i \Hom_R(M,N_i) \arrow[r] \arrow[d]	& \Hom_R(M,\colim_i N_i) \arrow[d]\\
			\colim_i\Hom_{R'}(M_{R'},(N_i)_{R'}) \arrow[r,"\simeq"]	\arrow[d, shift right] 
				\arrow[d, shift left]& \Hom_{R'}(M_{R'},\colim_i (N_i)_{R'}) \arrow[d, shift right] 
				\arrow[d, shift left]\\
			\colim_i\Hom_{R'}(M_{R'},(N_i)_{R''})	 \arrow[r,"\simeq"] & \Hom_{R'}(M_{R'},\colim_i (N_i)_{R''})
		\end{tikzcd}
	\end{equation}	
	Since a faithfully flat morphism is pure, 
	by \eqref{eq: pure descent commutative diagram} of~(\ref{lemma: descent = pure}), 
	the two columns are equalizer diagrams. Since $M_{R'}$ is finitely presented over $R'$ by assumption, 
	the two lower horizontal maps are bijections. It follows that the upper horizontal map is an isomorphism, 
	and hence $M$ is finitely presented. 
	
	Similarly, for finite generation, 
	it is enough to show that $\Hom_R(M,\vbl)$ commutes with filtered colimits of monomorphisms. 
	The argument then goes as above, noting that monomorphy is preserved by flat base change.
\end{proof}

\begin{proposition}\label{pro:module-finiteness-and-extensiveness}
	Let $R=\prod_{i=1}^n R_i$, where the $R_i$ are semirings, and let $M$ be an $R$-module.
	Let $M_i=R_i\otimes_R M$.
	Then $M$ is finitely generated (resp.\ presented) if and only each $M_i$ is finitely generated 
	(resp.\ presented) as an $R_i$-module.
\end{proposition}
\begin{proof}
	(Only if): Tensoring preserves finite generations and presentation.
	
	(If): Since $M\cong M_1\times \cdots \times M_n$, 
	it is enough to show each $M_i$ is finitely generated (resp.\ presented). 
	Further, by~(\ref{pro:direct image finiteness}), 
	we only need to show that each $R_i$ is finitely presented as an $R$-module.
	But this holds because each $R_i$ is the quotient of $R$ in the category of $R$-modules
	defined by the finitely many relations $e_j\sim\delta_{ij}$, 
	where $e_j$ denotes the $j$-th primitive idempotent in $R$ and $\delta$ denotes the Kronecker 
	delta function.
\end{proof}

\begin{remark}
Because of the two previous propositions,
we can say that finite presentation (or generation) is an fpqc-local property.
In the language of stacks, one would say that on $\Aff_{R}:=\CAlg{R}^{\op}$,
the full subcategories of $\Mod(A)$, for $A\in\CAlg{R}$, spanned by the 
finitely presented (or generated) $A$-modules form a substack of $\Mod_R$ in the fpqc topology. 
The analogous result is true for algebras:
\end{remark}

\begin{proposition}\label{proposition: fpqc-locally f.gen implies f.gen}
	Let $R$ be a semiring and $R'$ be faithfully flat $R$-algebra. Let $A$ be an $R$-algebra such that $A_{R'}$ is a finitely presented (resp.\ finitely generated) $R'$-algebra. Then $A$ is finitely presented (resp.\ finite generated) as an $R$-algebra. 
\end{proposition}
\begin{proof}
One may apply the same proof as in~(\ref{proposition: fpqc-locally f.pres implies f.pres})
but in the category of $R$-algebras instead of $R$-modules.
\end{proof}

\begin{proposition}\label{pro:algebra-finiteness-and-extensiveness}
	Let $R=\prod_{i=1}^n R_i$, where the $R_i$ are semirings, and let $A$ be an $R$-algebra.
	Let $A_i=R_i\otimes_R A$.
	Then $A$ is finitely generated (resp.\ presented) if and only each $A_i$ is finitely generated (resp.\ 
	presented) as an $R_i$-algebra.
\end{proposition}
\begin{proof}
Again, one may apply the same proof as in~(\ref{pro:module-finiteness-and-extensiveness})
but in the category of $R$-algebras instead of $R$-modules, observing that the same relations
$e_j\sim\delta_{ij}$ define $R_i$ is as a quotient of $R$ in the category $R$-algebras.
\end{proof}

\section{Projective modules}
\label{section: projective modules}

\begin{definition}
	Let $R$ be a semiring. An $R$-module $P$ is said to be \emph{projective} if the following equivalent conditions hold:
	\begin{enumerate}
		\item 
		There exists a free module $F$ and a surjection $p:F \to P$ with a section $s$. 
		\item 
		The functor $\Hom(P,-)$ preserves surjectivity. 
	\end{enumerate}
\end{definition}

\begin{proposition}\label{proposition: fg proj is fp}
If $M$ is a finitely generated projective $R$-module, then $M$ is finitely presented. 
\end{proposition}
\begin{proof}
Since $M$ is finitely generated and projective, we have a surjection $p:R^n \to M$ with a section $s$. We claim that $M$ is the coequalizer of $\id, s\circ p :R^n \to R^n$. Clearly, we have $p\circ \id = p\circ (s\circ p)$. In fact, suppose that we have the following diagram (of solid arrows):
\begin{equation}\label{eq: coeq}
\begin{tikzcd}
	R^n \arrow[r, shift right,swap,"s\circ p"] \arrow[r,shift left,"\id"] & R^n \arrow[r,"p"] \arrow[d,swap,"q"] & M \arrow[dl, dashed,"r=q\circ s"]\\
	& N & 
\end{tikzcd}
\end{equation}
With $r=q\circ s$, the diagram commutes since $s$ is a section of $p$. Moreover, since $p$ is surjective, such $r$ is unique once it exists. This shows that \eqref{eq: coeq} is a coequalizer diagram, and hence $M$ is finitely presented. 
\end{proof}

\begin{proposition}\label{proposition: proj is flat}
Projective modules are flat.
\end{proposition}
\begin{proof}
Indeed, every projective module $M$ is a filtered colimit of free modules in the following way.
Take a free module $F$ and a surjection $p:F \to M$. Thus $p$ has a section $s$ since $M$ is projective.
Then $M$ is the colimit of the following diagram:
\begin{equation*}\label{eq: filtered colimit}
\begin{tikzcd}[row sep=large, column sep=large]
F
\arrow[loop, from=1-1, to=1-1, out=30,in=330,looseness=6,min distance=7mm,"\id"]
\arrow[loop, from=1-1, to=1-1, out=210,in=150,looseness=6,min distance=7mm,"sp"]
\end{tikzcd}
\end{equation*}
But this diagram is filtered, since $sp\circ sp =sp\circ\id$, and hence $M$ is a filtered colimit of free modules. It follows from~(\ref{thm: equivalent conditions for flatness}) that $M$ is flat. 
\end{proof}

\begin{proposition} \label{proposition: f.pres flats are projective}
Finitely presented flat modules are projective.
\end{proposition}
\begin{proof}
Choose a presentation:
\begin{equation}
\begin{tikzcd}
R^m \arrow[r, shift right,swap,"q"] \arrow[r,shift left,"p"] & R^n \arrow[r,"f"] & M, \quad f=\textrm{coeq}(p,q).
\end{tikzcd}
\end{equation}
Since $M$ is flat, from~(\ref{thm: equivalent conditions for flatness}), the category $FP_M$ as in~(\ref{subsection: filtered colimits}) has a cofinal family of the finite free modules. Therefore, there exists a free module $F$ and maps as in the diagram
\begin{equation}
\begin{tikzcd}
R^m \arrow[r, shift right,swap,"q"] \arrow[r,shift left,"p"] & R^n \arrow[r,"f"] \arrow[d,"g"] & M \arrow[dashed, out=-90, in=-20,  dl,"s"]\\
& F \arrow[ru,swap, "h"] & 
\end{tikzcd}
\end{equation}
such that $gp = gq$ and $f=hg$. Since $M$ is the coequalizer of $p$ and $q$, there exists $s$ such that $sf = g$. It follows that
\begin{equation}
	h sf = h  g = f = 1_M f. 
\end{equation}
Since $f$ is surjective, we have $h s = 1_M$. Therefore $M$ is projective. 
\end{proof}

\begin{corollary}\label{corollary: equi flat proj}
For finitely presented modules, flatness is equivalent to projectivity. 
\end{corollary}
\begin{proof}
	By~(\ref{proposition: proj is flat}) and~(\ref{proposition: f.pres flats are projective}).
\end{proof}

\begin{example} \label{ex:MacPherson}
\emph{Finite non-projective $\mathbb{B}$-modules.}
A.~W.~Macpherson proves in \cite{macpherson2019projective}, theorem 3.11, that a finite $\mathbb{B}$-module $M$
is projective if and only if it has unique irredundant primitive decompositions. Here, an element $x \in M$ is
said to be primitive if whenever $x =\sum_i x_i$, then $x=x_i$ for some $i$. Since $M$ is finite, any element of
$M$ can be written as a sum of primitive elements. A decomposition $x=\sum_i x_i$ is said to be 
primitive if the summands
$x_i$ are primitive, and is said to be irredundant if there are no order relations among different $x_i$;
equivalently, the family $(x_i)_i$ of terms in the decomposition is minimal.

A $\bB$-module is equivalent to an upper semilattice (following the terminology of Lurie~\cite{Lurie:categorical-logic-classt}), 
which is to say a partially ordered set such that
every finite set of elements (empty or not) has a supremum. Consider the $\bB$-module
	$$
	M=\bigoplus_{i=1}^3\bB x_i/(x_1+x_2\sim x_1+x_3 \sim x_2+x_3).
	$$
One checks $M$ is obtained from the following upper semilattice:
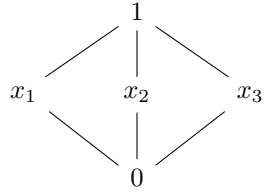
\[
\begin{tikzcd}[every arrow/.append style={dash}]
	& 1  \arrow[d] \arrow[dl] \arrow[dr] &\\
	x_1 \arrow[dr] & x_2 \arrow[d] & x_3 \arrow[dl]\\
	&0 & 
\end{tikzcd}
\]
So the $x_i$ are primitive and distinct. Because of the equalities
\[
1=x_1 + x_2 = x_1 + x_3 = x_2+x_3,
\]
the element $1$ has three different irredundant primitive decompositions, and hence by MacPherson's theorem,
$M$ is not projective. 
	
\end{example}

\begin{theorem}\label{thm:fg-projective-Rplus-modules-are-free2}
	Let $R$ be either $\bN$ or $F\cap\Rpos$, where $F$ is a subfield of $\bR$.
	Then every finitely generated projective $R$-module $M$ is free.
\end{theorem}
\begin{proof}
	Let $x_1,\dots,x_n$ be a minimal family of generators of $M$. 
	Such a family exists because $M$ is finitely generated.
	Consider the surjective morphism $p:R^n\to M$, given by $p(e_i)=x_i$.
	Because $M$ is projective, $p$ has a section $s:M\to R^n$.
	Write $s(x_i)=(r_{i1},\dots,r_{in})$.
	Then we have
		$$
		x_i = p(s(x_i)) = \sum_j r_{ij}x_j.
		$$

	Let us show $r_{ii}\geq 1$. 	
	If not, then since $M$ is additively cancellative (being a submodule of $R^n$ via $s$), 
	we would have $(1-r_{ii})x_i = \sum_{j\neq i}r_{ij}x_j$.
	But since $0<1-r_{ii}\leq 1$, 
	we have $1-r_{ii}\in R^\times$, because $R$ is either $\bN$ or a semifield.
	Hence we could write $x_i$ as a linear combination of the $x_j$ for $j\neq i$,
	which would contradict the minimality of the generating set. Therefore $r_{ii}\geq 1$.
	
	So, again since $M$ is additively cancellative, we have
		$$
		0 = (r_{ii}-1) x_i + \sum_{j\neq i} r_{ij} x_j.
		$$
	Since $M$ is negative free, we have $(r_{ii}-1) x_i = 0$ and $r_{ij}x_j = 0, j\neq i$.
	But each $x_j$ freely generates its span $Rx_j$, since each $x_j$ is nonzero and $M$ is a submodule of $R^n$.
	Therefore we have $r_{ii}=1$ and $r_{ij}=0$ for $j\neq i$.
	This means $s(x_i)=e_i$ for each $i$.
	It follows that $s$ is surjective, and hence an isomorphism.
\end{proof}

\begin{example}
We can now give many examples of finitely generated sub-$\Rpos$-modules $M$ of $\bR^n$ 
which are not projective.
Let $M$ be the $\Rpos$-cone in the half-space $\bR^{n-1}\times\Rpos$ with cross section $C$, 
where $C$ is a convex subset of $\bR^{n-1}\times \{1\}$. 
If $C$ is the convex hull of a finite set, then $M$ will be finitely generated (and only then).
But by~(\ref{thm:fg-projective-Rplus-modules-are-free2}), it will not be projective unless
$C$ is a simplex.
So for example the square cone is not projective, as was already established in~(\ref{eg:square-cone}).
\end{example}

\begin{remark}
	We note that the basis given by the proof of the theorem above, or any basis of a free module over $R$,
	is unique up to scalar multiplication. 
	Or equivalently, $\GL_n(R)=(R^\times)^n\rtimes S_n$. 
	(Proof: Since $R$ is negative free and has no zero divisors, the function $f\:R\to\bB$ with $f(r)=1$ 
	if and only if $r\neq 0$ is a homomorphism. Consider the induced map $\GL_n(R)\to \GL_n(\bB)$. Since 
	$\GL_n(\bB)$ coincides with the set of permutation matrices, and since $f^{-1}(0)=\{0\}$, 
	every element of $\GL_n(R)$ must be a permutation matrix times a diagonal matrix, 
	the diagonal entries of which are necessarily invertible.)
\end{remark}

\begin{remark}
The argument in the theorem above goes through slightly more generally.
It is enough for $R$ to be a negative-free, additively cancellative semiring which is
totally ordered by the relation $[x\leq y]:=[\exists z.\; x+z=y]$ 
and satisfies $0<x\leq 1 \;\Rightarrow\; x\in R^\times.$

The semirings $\bB$ and $A$ of (\ref{eg:weakly-connected}) are negative free but do not satisfy
all of the other properties. And indeed, 
there are finitely generated projective modules over each of them which are not free, or even locally free.
See (\ref{ex:monotone-modules-are-projective-over-B}) and (\ref{eg:weakly-connected}).
\end{remark}

\begin{example}
	\emph{A flat non-projective $\Rpos$-module.}
	Let $M$ denote the open quadrant
	$\{(x,y)\in \Rpos^2\mid x,y>0 \text{ or } x=y=0\}$.
	Then $M$ is a sub-$\Rpos$-module of $\Rpos^2$.
	It is flat over $\Rpos$ because it is an increasing union of free modules of rank $2$:
	$$M=\colim_{\varepsilon\to 0}\Rpos(1,\varepsilon)\oplus\Rpos(\varepsilon,1).$$
	To see $M$ is not projective, suppose there were a set $S$ and a map $\Rpos^{(S)}\to M$ with a section $s$.
	Choose two linearly independent vectors $v,w\in M$. 
	Then $s(v)$ and $s(w)$ would have support contained in a finite subset $T\subseteq S$. 
	Since $v$ and $w$ span $\bR^2=\bR\otimes_{\Rpos} M$ over $\bR$,
	the image of the base change $\bR\otimes_{\Rpos} M\to \bR^{(S)}$ would lie
	in $\bR^{T}$. Therefore the image of the original map $s$ would lie in $\bR^T\cap \Rpos^{(S)} = \Rpos^T$.
	This would imply that the composition $\Rpos^T\to\Rpos^{(S)}\to M$ is surjective,
	which would contradict the fact that $M$ is not finitely generated.
\end{example}

\section{Dualizable modules}
\label{section: dualizable modules}

In what follows, let $R$ be a semiring. 
Recall that a cocontinuous functor is a functor which preserves small colimits.

\begin{definition}
An $R$-module $M$ is \emph{dualizable} if the functor $$\Hom_{\Mod(R)}(M,-):\Mod(R) \to \Mod(R)$$ is cocontinuous.
\end{definition}

We will show below that $M$ is dualizable if and only if it has a dual in the sense of the theory of monoidal categories. We prefer the definition above because it makes it obvious that dualizability is property, as opposed to a structure.

For any $R$-module $M$, dualizable or not, we define the \emph{dual module}
	$$
	M^\vee = \Hom_R(M,R).
	$$

\begin{proposition}\label{pro:tensor-with-dual}
Let $M$ be a dualizable $R$-module. Then  $$ M^\vee \otimes_R N\longisomap\Hom_{\Mod(R)}(M,N).$$ 
\end{proposition}
\begin{proof}
Indeed, for $N=R$, it is clear.
Write $G(N)=\Hom_{\Mod(R)}(M,N)$,
Since $M$ is dualizable, $G$ is cocontinuous. Therefore for any $R$-module $N$, we have
\begin{multline*}
G(N)=G(\underset{R\to N}{\colim}~R)=\underset{R\to N}{\colim}~G(R)=\underset{R\to N}{\colim}~M^\vee \\
=\underset{R\to N}{\colim}~(M^\vee\otimes_R R) 
=M^\vee  \otimes_R \underset{R\to N}{\colim}~R = M^\vee\otimes_R N. 
\end{multline*}
\end{proof}

\subsection{Triangle axioms}\label{subsection: triangle axioms}
Our notion of dualizability agrees with that in the theory of monoidal categories, as for instance
found in the book of Etingof--Gelaki--Nikshych--Ostrik~\cite{Etingof2016tensor-categories}, as we now explain.

If $M$ is dualizable, then the unit and counit maps of the tensor-hom adjunction at the $R$-module $R$
\begin{equation}
\label{unit-counit-maps2}
	\eta_R\:R \longmap \Hom_R(M,M\otimes_R R), \quad \varepsilon_R\:M\otimes \Hom_R(M,R) \longmap R	
\end{equation}
give rise by~(\ref{pro:tensor-with-dual}) to maps purely expressed in terms of the monoidal
structure $\otimes_R$:
\begin{equation}
\label{unit-counit-maps}
\eta:R \to M^\vee\otimes_RM, \quad \varepsilon:M\otimes_RM^\vee \to R. 
\end{equation}
The triangle axioms for adjoint functors imply the triangle axioms for duals, namely the commutativity of the following diagrams:
\begin{equation}
\begin{tikzcd}[row sep=1cm, column sep = 1.3cm]\label{eq: triangle1}
	M\otimes_RR \arrow[r,"1_M\otimes_R\eta"] \arrow[d,swap,"\simeq"]  & M\otimes_RM^\vee\otimes_RM \arrow[dl,"\varepsilon\otimes_R1_M"] \\
	R\otimes_RM
\end{tikzcd}
\end{equation}
\begin{equation}
\begin{tikzcd}[row sep=1cm, column sep = 1.3cm]\label{eq: triangle2}
	R\otimes_RM^\vee \arrow[r,"\eta\otimes_R1_{M^\vee}"] \arrow[d,swap,"\simeq"]  & M^\vee\otimes_RM\otimes_RM^\vee \arrow[dl,"1_{M^\vee}\otimes_R\varepsilon"] \\
	M^\vee\otimes_RR
\end{tikzcd}
\end{equation}
(See~\cite{Etingof2016tensor-categories}, p.\ 40, definition 2.10.2.) Thus $M^\vee$ is a right dual of $M$,
and also a left dual because $\otimes$ is symmetric.

Conversely, suppose $M$ has a right dual $M^*$, meaning there are maps $\eta\:R\to M^*\otimes_R M$ and $\varepsilon\:M\otimes M^*\to R$ such that the triangle axioms (\ref{eq: triangle1})--(\ref{eq: triangle2}) are satisfied but with $M^*$ instead of $M^\vee$.
Then for any $R$-module $N$, we can apply $\vbl\otimes_R N$ to the maps $\eta$ and $\varepsilon$ and hence
obtain an adjunction $M^*\otimes_R\vbl \vdash M\otimes_R\vbl$.
Therefore $\Hom_R(M,\vbl)$ agrees with $M^*\otimes_R\vbl$, and hence it preserves all colimits.
Therefore $M$ is dualizable.

\begin{proposition}\label{proposition: dualizability is an fpqc-local property}
	Dualizability is an fpqc-local property.
\end{proposition}
\begin{proof}
	The triangle axioms above are expressed in fpqc-local terms.
\end{proof}

Once again, in the language of stacks, one would say that dualizable modules form a substack of
the stack of all modules in the fpqc topology.

\begin{proposition}\label{proposition: dualizable iff f.g. proj.}
$M$ is dualizable if and only if $M$ is a finitely generated projective module. 
\end{proposition}

\begin{proof}
We copy the argument for rings from the nLab~\cite{nlab:dualizable-modules}.

(If): Note that $\Mod(R)$ is a closed symmetric monoid category, and in a closed symmetric monoidal category, dualizable objects are closed under retractions and direct sums. Since $R$ itself is dualizable, it follows that $M$ is dualizable. 

(Only if): With the same notation as in (\ref{subsection: triangle axioms}), write $\eta(1_R)=\sum_{i=1}^n f_i\otimes m_i$, with $m_i\in M$, $f_i\in \Hom(M,R)$. 
Then from the triangle identities, for each $m \in M$ and $\varphi \in \Hom(M,R)$, we have
\begin{equation}\label{eq: 1}
\sum_{i=1}^nm_if_i(m)=m,
\end{equation}
\begin{equation}\label{eq: 2}
\sum_{i=1}^n\varphi(m_i)f_i=\varphi.
\end{equation}
Now consider the map $f:R^n \to M$ sending $(r_i)$ to $\sum_{i=1}^nr_im_i$. By (\ref{eq: 1}), this map is surjective and hence $M$ is finitely generated. Also consider 
the map $g:M \to R^n$ sending $m \in M$ to $(f_i(m))$. From \eqref{eq: triangle1}, we have that $fg=1_M$, showing that $M$ is projective.
\end{proof}

\section{Locally free modules}
\label{section: locally free modules}

\begin{definition}\label{def:fpqc-locally-free}
An $R$-module $M$ is \emph{fpqc-locally free} if there exist a finite family of flat $R$-algebras $(R_i)_{i\in I}$ such that $\prod_i R_i$ is faithfully flat over $R$ and each 
$R_i\otimes_R M$ is a free $R_i$-module. It is \emph{fpqc-locally finite free} if there exists such a family
for which each $R_i\otimes_R M$ is free of finite rank.
Similarly, it is \emph{Zariski-locally (finite) free} if each $R_i$ can be taken to be of 
the form $R[1/a_i]$.
\end{definition}

\begin{example}
\label{ex:locally frees over Booleans}
Every Zariski-locally free module over a local semiring is free.

Every fpqc-locally free $\bB$-module $M$ is free. Indeed, by Golan's theorem~(\ref{thm:Golan}),
for any faithfully flat map $\mathbb{B}\to \prod_i R_i$ as above, there exists an $i_0$ such that
there exists a map $R_{i_0}\to \bB$.
Therefore if each $R_i\otimes_\bB M$ is free, then so is $M=\bB\otimes_{R_{i_0}} R_{i_0}\otimes_\bB M$.

More generally, over $\bB$ any locally trivial object of any type is trivial. This was used in Culling's thesis (example 4.20 of \cite{culling2019etale}) to show that the \'etale fundamental group of $\Spec(\bB)$ is trivial.
\end{example}

\begin{proposition}\label{pro: locally frees are dualizable}
	Every fpqc-locally finite free module is dualizable.
\end{proposition}

\begin{proof}
	This follows from~(\ref{proposition: dualizability is an fpqc-local property}),
	since all free modules of finite rank are dualizable.
\end{proof}

\begin{corollary}\label{proposition: fpqc-locally free is f.pres.proj}
Every fpqc-locally finite free $R$-module is finitely presented and projective. 
\end{corollary}
\begin{proof}
(\ref{pro: locally frees are dualizable}) and (\ref{proposition: dualizable iff f.g. proj.}).
\end{proof}

\begin{example}
The converse of~(\ref{proposition: fpqc-locally free is f.pres.proj}) holds over $\Rpos$, 
by~(\ref{thm:fg-projective-Rplus-modules-are-free2}),
but it does not hold in general. Indeed, it was already observed in~\cite{jun2024vector}, example 3.21, 
that the $3$-element $\bB$-module is projective but not free and hence not Zariski-locally free. 
Slightly more generally, we have the following.
\end{example}

\begin{example} \label{ex:monotone-modules-are-projective-over-B}
Let $P$ be a finite partially ordered set, and let $A$ denote the $\bB$-algebra of monotone functions $P\to\bB$.
It is finitely generated. We will show that $A$ is projective over $\bB$ but unless $P$ is discrete (that is, satisfies $x\leq y \Leftrightarrow x=y$), 
$A$ is not free and hence not fpqc-locally free, by~(\ref{ex:locally frees over Booleans}).

To see it is projective, consider the $\bB$-module map $p\:\bB^P\to A$ determined by
$e_x\mapsto f_x$, where $f_x(y)$ is $1$ if $y\geq x$ and is $0$ otherwise. Then $p$ is a retraction 
onto the subalgebra $A\subseteq \bB^P$: for any $g\in A$, we have $g=\sum_{g(x)=1} e_x$ and hence 
	$$
	p(g)=\sum_{g(x)=1} p(e_x) = \sum_{g(x)=1}f_x.
	$$
Therefore we have $p(g)(y)=1$ if and only if there exist $x\leq y$ such that $g(x)=1$, 
which holds if and only if $g(y)=1$. Therefore $p(g)=g$.

Second, observe that $f_x$ is primitive in the sense of~(\ref{ex:MacPherson}). 
Indeed, if $f_x = \sum_i g_i$, 
then we have $1=f_x(x)=\sum_i g_i(x)$, and hence there exists an $i_0$ such that $g_{i_0}(x)=1$. 
Therefore $g_{i_0}(y)=1$ for all $y\geq x$ and hence $g_{i_0}\geq f$. But $g_i\leq f_x$ for any $i$. 
Therefore $g_{i_0}=f_x$

Now suppose $A$ is freely generated by a subset $S\subseteq A$. 
Since any $f_x$ is primitive, it is must be an element of $S$.
If there exist two comparable elements $x\leq y$, then we have $f_y\leq f_x$. 
It follows that $f_y=f_x$ and hence $y=x$.
\end{example}

\begin{remark}
\label{rmk:monotone-R_+-functions}
We note that the analogous construction over $\Rpos$ might or might not be projective. So let $P$
again be a finite partially ordered set, but now let $A$ be the $\Rpos$-algebra of monotone functions
$P\to\Rpos$. If $P$ is discrete, then we have $A=\prod_P\Rpos$, which is free as an $\Rpos$-module. If $P$ is
not discrete, then $A$ still
might or might not be free, or equivalently projective by~(\ref{thm:fg-projective-Rplus-modules-are-free2}). 

A free example is given by $P=\{0\to 1\}$: then $A$ is free with basis $(0,1),(1,1)$.

For a non-free example, let $P$ be the partially ordered set
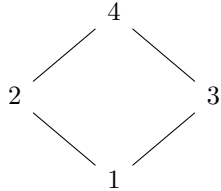
\[
\begin{tikzcd}[every arrow/.append style={dash}]
	& 4  \arrow[dl] \arrow[dr] &\\
	2 \arrow[dr] && 3 \arrow[dl] \\
	&1
\end{tikzcd}
\]
For any $x\in P$, consider the function $f_x\:P\to\Rpos$ defined above. Then we have
$$
f_1 = (1,1,1,1),\quad f_2=(0,1,0,1),\quad f_3=(0,0,1,1),\quad f_4=(0,0,0,1).
$$
Each of these elements is $\Rpos$-linearly primitive. 
But there is another extremal ray, given by the span of the primitive element
$$
g=(0,1,1,1).
$$
One checks these five elements generate $A$ as an $\Rpos$-module. 
But observe that they are not independent, because of the relation $g+f_4=f_2+f_3$. 
So $A$ is not free over $\Rpos$, and hence nor is it projective or flat.
\end{remark}

\begin{example} \label{eg:weakly-connected}
The divergence between projective modules and locally free modules is not restricted to
algebras over $\bB$ or other semirings which are not additively cancellative.
There are in fact examples over flat $\Rpos$-algebras. 
So it is far from a pathological phenomenon.

For instance, let $A=\{(b,c)\in\Rpos\times\Rpos\mid b\leq c\}$. 
So $A$ is the set of monotone functions $\bB\to\Rpos$ considered above.	
Then $A$ is free of rank $2$ over $\Rpos$, and is in particular flat.
	
Let $t$ denote the idempotent $(0,1)\in A$.
Then $A$ agrees with $\Rpos[t]/(t^2\sim t)$ since the composition
$$\Rpos\oplus \Rpos t \longisomap \Rpos[t]/(t^2\sim t) \longlabelmap{t\mapsto (0,1)} A$$
is an isomorphism.
Consider the $A$-algebra $A[1/t]=A/(t\sim 1)=\Rpos$. 
Since $t$ is idempotent, $A[1/t]$ is isomorphic as an $A$-module to the ideal $(t)$, 
which is finitely generated and projective as an $A$-module.
However its rank jumps: at $t=0$, the rank is $0$,
but at $t=1$, it is $1$. 

But this implies $A$ is not locally free. 
Indeed, $A$ is a local semiring: if $a_1+a_2$ is a unit, then for elementary reasons each $a_i$ is of
the form $(b_i,b_i)$ and one of the $b_i$ must be nonzero, and hence $a_i$ will be a unit. 
Therefore if $\prod_{i=1}^m A[1/a_i]$ is a Zariski cover, then some factor 
$A[1/a_i]$ must agree with $A$, and hence if $A[1/t]$ becomes free over a Zariski cover, it must already have been free over $A$, which it is not. 
\end{example}

\section{Line bundles}
\label{section: line bundles}

The purpose of this section is to show that all the usual definitions of \emph{line bundle} over a ring
remain equivalent over all semirings.

\begin{definition}\label{definition: invertible module}
$M$ is \emph{invertible} if the functor $M\otimes_R\vbl:\Mod(R) \to \Mod(R)$ is an equivalence.
\end{definition}

One could equivalently require that its right adjoint $\Hom_R(M,\vbl)$ be an equivalence.
This definition agrees with the one found in Etingof \emph{et al.}\  
\cite{Etingof2016tensor-categories} (p.~43, definition 2.11.1), by the following.

\begin{proposition}
	\label{pro:invertible-module-equivalent-conditions}
An $R$-module $M$ is invertible if and only if it is dualizable and the unit map $\eta$ and counit map $\varepsilon$ are isomorphisms.
\end{proposition}
\begin{proof}
If the functor $M\otimes_R\vbl$ is an equivalence, 
then its right adjoint $\Hom_R(M,\vbl)$ is an equivalence and hence cocontinuous.

Conversely, suppose $M$ is dualizable. Then the functor $M^\vee\otimes_R\vbl$ is the right adjoint of
$M\otimes_R\vbl$. If the unit (and counit) of $M$ is an isomorphism of modules, then the unit (and counit)
of the adjunction is an isomorphism of functors. Therefore $M^\vee\otimes_R\vbl$ is a quasi-inverse of
$M\otimes_R\vbl$, and hence the latter is an equivalence.
\end{proof}

\begin{corollary}\label{cor: invertibility is an fpqc-local property}
	Invertibility is an fpqc-local property.
\end{corollary}
\begin{proof}
By the previous proposition, an $R$-module $M$ is invertible if and only if the maps 
$\varepsilon:M^\vee\otimes_RM \to R$ and $\eta:R \to M\otimes_RM^\vee$ are isomorphisms satisfying the triangle 
identities \eqref{eq: triangle1} and \eqref{eq: triangle2}. But these are fpqc-local properties of morphisms.
So invertibility is also an fpqc-local property.
\end{proof}

In other words, invertible modules form a stack in the fpqc topology.
Also, it follows that every locally free module of rank $1$ is invertible.

\begin{definition}\label{definition: projection}
Let $P_n(R)$ denote $\{\pi \in R^{n\times n} \mid \pi^2=\pi\}$, the set of $n\times n$ projector matrices. 
For each projector $\pi \in P_n(R)$, let $M_\pi\subseteq R^n$ 
denote the image of $\pi\:R^n\to R^n$, let $p_\pi\:R^n\to M_\pi$ denote the same map $\pi$ but with the codomain 
restricted to $M_\pi$, and let $s_\pi\: M_\pi\to R^n$ denote the inclusion.
\end{definition}

We therefore have the epi-mono factorization of $\pi$:
	$$
	\begin{tikzcd}
	R^n 	\arrow[r,->>,swap, "p_\pi"] & M_\pi \arrow[hook, bend right,  l,swap, "s_\pi"]
	\end{tikzcd}
	$$
For any projective module $M$ generated by $n$ elements, there is a projector $\pi\in P_n(R)$ such that
$M\cong M_\pi$.

\begin{proposition}\label{proposition: commutative diagram prop}
For any $\pi$, the following diagrams commute:
\begin{equation}
	\begin{tikzcd}[row sep=1cm]
R^n \otimes (R^n)^\vee \arrow[r,"\varepsilon"] & R & R^n \otimes (R^n)^\vee \arrow[d,swap, "p_\pi\otimes s_\pi^\vee"]  & R \arrow[l,swap,"\eta"] \arrow[dl,"\eta_\pi"]\\
M_\pi\otimes M_\pi^\vee \arrow[u,"s_\pi\otimes p_\pi^\vee"] \arrow[ur,swap,"\varepsilon_\pi"] & & M_\pi\otimes M_\pi^\vee.
	\end{tikzcd}
\end{equation}
\end{proposition}
\begin{proof}
The diagrams are dual to each other. So it is enough to consider the first.	
For any $m\otimes \chi \in M_\pi \otimes M_\pi^\vee$, we have
\begin{multline*}
\varepsilon ( (s_\pi\otimes p_\pi^\vee)(m\otimes \chi))= \varepsilon ( s_\pi(m)\otimes (\chi\circ p_\pi))=(\chi\circ p_\pi)(s_\pi(m)) \\
=\chi(p_\pi(s_\pi(m)))=\chi(m)=\varepsilon_\pi(m\otimes \chi). 	
\end{multline*}
\end{proof}

\begin{proposition}\label{pro:unit-counit-criterion}
	Consider a projector $\pi\in P_n(R)$.
	\begin{enumerate}
		\item The counit $\varepsilon_\pi$ is surjective if and only if the matrix entries $\pi_{ij}\in R$ generate the unit ideal.	
		\item The unit $\eta_\pi$ is surjective if and only if there exists a matrix $z$ such that 
		$z_{ij}\pi_{kl} = \pi_{kj}\pi_{il}$ for all $i,j,k,l$.
		\item If both $\varepsilon_\pi$ and $\eta_\pi$ are surjective, we have
			$\pi_{ij}\pi_{kl}=\pi_{il}\pi_{kj}$ for all $i,j,k,l$.
		(``All $2\times 2$ submatrices of $\pi$ are singular.'')
	\end{enumerate}
\end{proposition}
\begin{proof}
First consider the following diagram, which commutes by~(\ref{proposition: commutative diagram prop}):
\begin{equation}
	\label{diag:projector}
\begin{tikzcd}[column sep=0.8cm]
R \ar[r,hook,"\eta"] \ar[dr,swap,"\eta_\pi"]
	& R^n \otimes (R^n)^\vee \ar[d,->>,"p\otimes s^\vee"] \ar[dr,"\pi\otimes\pi^\vee"]\\
& M\otimes M^\vee \ar[r,hook,"s\otimes p^\vee"] \ar[dr,swap,"\varepsilon_\pi"]
	& R^n\otimes(R^n)^\vee \ar[d,->>,"\varepsilon"] \\
& & R
\end{tikzcd}
\end{equation}
Observe that we have
\begin{equation}
	\label{eq:pi-map}
\pi\otimes \pi^\vee\: e_i\otimes e_j^\vee \mapsto 
	\sum_{k,l} \pi_{ki} e_k \otimes \pi_{jl}e_l^\vee \\= \sum_{k,l} \pi_{ki}\pi_{jl} (e_k\otimes e_l^\vee).	
\end{equation}
It follows that for an arbitrary element $b=\sum_{i,j}b_{ij} e_i\otimes e_j^\vee$, we have
\begin{equation}
	\label{eq:pi-map2}
(\pi\otimes\pi^\vee)(b) = \pi b \pi,
\end{equation}
where the right-hand side is the matrix product and both sides are interpreted as matrices.

(1): Since $p\otimes s^\vee$ is surjective,
$\varepsilon_\pi$ is surjective if and only if $\varepsilon\circ(\pi\otimes\pi^\vee)$ is surjective. And this holds
if and only if the elements $\varepsilon\circ(\pi\otimes\pi^\vee)(e_i\otimes e_j^\vee)$ generate the unit ideal.
But these elements are just the matrix entries, by (\ref{eq:pi-map}):
$$
\varepsilon\circ(\pi\otimes\pi^\vee)(e_i\otimes e_j^\vee) = \sum_k \pi_{ki}\pi_{jk} = (\pi^2)_{ji} = \pi_{ji}
$$

(2):
The map $\eta_\pi$ is surjective if and only if there exists a map $z\:R^n\otimes (R^n)^\vee\to R$ such that $\eta_\pi\circ z = p\otimes s^\vee$. Since $s\otimes p^\vee$ is injective, this condition is equivalent to 
$$
(\pi\otimes \pi^\vee)(\eta(z(a))) = (\pi\otimes\pi^\vee)(a)\quad\text{for all matrices }a.
$$
By (\ref{eq:pi-map2}), and since $\pi^2=\pi$, this is equivalent to
$$
z(a) \pi = \pi a \pi \quad\text{for all matrices }a.
$$
Because this equation is linear in $a$, it is further equivalent to require it holds only for matrices of the form $a=e_i\otimes e_j^\vee$. This is in turn equivalent to 
$$
z_{ij} \pi = \pi_{\ast i}\pi_{j \ast}\quad\text{for all }i,j,
$$
where we write the element $z\in (R^n\otimes (R^n)^\vee)^\vee = (R^n)^\vee\otimes R^n$ as 
$z=\sum_{i,j}z_{ij} e_i^\vee\otimes e_j$ and where $\pi_{\ast i}$ denotes the $i$-th column of $\pi$ and $\pi_{j \ast}$ denotes the $j$-th row. Thus $\eta_\pi$ is surjective if and only if
there exists a matrix $z$ such that for all $i,j,k,l$, we have
$$
z_{ij}\pi_{kl} = \pi_{ki}\pi_{jl}.
$$

(3): By (1), the matrix entries $\pi_{ij}$ generate the unit ideal. Since the condition to be proved can be checked locally, we may pass to a cover and assume one of the matrix entries is a unit, say $\pi_{i'j'}$. Then by (2), there is a matrix $z$ such that for all $k,l$, we have
$$
z_{kl}\pi_{i'j'}=\pi_{i'l}\pi_{kj'}.
$$
Taking $k=i'$ and $l=j'$, we see $z_{i'j'}=\pi_{i'j'}$, since $\pi_{i'j'}$ is a unit.
Therefore, for all $k,l$, we have $\pi_{i'j'}\pi_{kl}=\pi_{i'l}\pi_{kj'}$ and hence  $\pi_{kl}=\pi_{i'j'}^{-1}\pi_{i'l}\pi_{kj'}$. 
It follows that for all $i,j,k,l$, we have
	$$
	\pi_{ij}\pi_{kl} = 
	(\pi_{i'j'}^{-1}\pi_{i'j}\pi_{ij'})
	(\pi_{i'j'}^{-1}\pi_{i'l}\pi_{kj'})
	$$
and
	$$
	\pi_{il}\pi_{kj} =
	(\pi_{i'j'}^{-1}\pi_{i'l}\pi_{ij'})
	(\pi_{i'j'}^{-1}\pi_{i'j}\pi_{kj'}),
	$$
the right-hand sides of which agree. Therefore $\pi_{ij}\pi_{kl} = \pi_{il}\pi_{kj}$.
\end{proof}

\begin{theorem}\label{thm:inv=loc-free-rank-1}
Every invertible module $M$ over a semiring $R$ is Zariski-locally free of
rank $1$. 
More precisely, if $\pi\in P_n(R)$ is a projector such that $M\cong M_\pi$, 
then $\prod_{i,j}R[\frac{1}{\pi_{ij}}]$ is a faithfully flat $R$-algebra over which $M$ becomes free of rank $1$.
\end{theorem}
\begin{proof}
First observe that faithful-flatness follows from~(\ref{pro:unit-counit-criterion})(1) and~(\ref{pro: fpqc is zariski}).
	
It remains to show that $M$ becomes free of rank $1$ over each $R[1/\pi_{ij}]$.
Since $M$ is invertible, its unit and counit are isomorphisms, by (\ref{pro:invertible-module-equivalent-conditions}).
Therefore by (\ref{pro:unit-counit-criterion})(3), we have $\pi_{kj}\pi_{il}=\pi_{ij}\pi_{kl}$ for all $k,l$, and hence 
\[
\begin{tikzcd}
\begin{bmatrix}
	\pi_{1j}\\
	\vdots\\
	\pi_{nj}
\end{bmatrix}  \begin{bmatrix}
\pi_{i1} \cdots \pi_{in}
\end{bmatrix} = \begin{bmatrix}
&\vdots & \\
\cdots &\pi_{kj}\pi_{il} & \cdots \\
& \vdots &
\end{bmatrix} = \pi_{ij}\begin{bmatrix}
&\vdots & \\
\cdots &\pi_{kl} & \cdots \\
& \vdots &
\end{bmatrix} =\pi_{ij}\pi
\end{tikzcd}
\]
It follows that over $R[1/\pi_{ij}]$ we have
\[
\begin{tikzcd}
\begin{bmatrix}
	\pi_{1j}\\
	\vdots\\
	\pi_{nj}
\end{bmatrix}  \begin{bmatrix}
	\frac{\pi_{i1}}{\pi_{ij}} \cdots 1 \cdots\frac{\pi_{in}}{\pi_{ij}}
\end{bmatrix} 	 = \pi. 
\end{tikzcd}
\]
Observe that each of the vectors on the left-hand side contains an invertible entry: 
$\pi_{ij}$ in the first and $1$ in the second.
Therefore in the corresponding factorization $\pi\:R^n\to R\to R^n$, the first map is surjective
and the second is injective. It follows $M_\pi$, the image of $\pi$, is isomorphic to $R$.
\end{proof}

\begin{corollary}\label{cor: invertible-iff-locally-rank-1}
Let $R$ be a semiring, and let $M$ be an $R$-module. Then the following are equivalent:
\begin{enumerate}
	\item $M$ is invertible,
	\item $M$ is fpqc-locally free of rank $1$, 
	\item $M$ is locally free of rank $1$ in the Zariski topology of 
		To\"en--Vaqui\'e~(\ref{rmk:Zariski topology on Aff}),
	\item $M$ is Zariski-locally free of rank $1$.
\end{enumerate}
\end{corollary}

\begin{remark}
Let $M$ be an invertible $R$-module.
It follows from~(\ref{cor: invertible-iff-locally-rank-1}) that $\textrm{Sym}^n(M)=M^{\otimes n}$.
Therefore the bi-infinite
tensor algebra $T^{\pm}(M)=\bigoplus_{n \in \mathbb{Z}}M^{\otimes n}$ is commutative. Then $M$ has a canonical trivialization over $T^\pm(M)$: 
\begin{equation}
T^\pm(M)\otimes_RM=\bigoplus_{n \in \mathbb{Z}}M^{\otimes (n+1)} \xrightarrow{\simeq}	\bigoplus_{n \in \mathbb{Z}} M^{\otimes m} = T^\pm(M). 
\end{equation}
This is in fact the universal pair consisting of an $R$-algebra $R'$ and a trivialization 
$R'\otimes_R M\isomap M$.

By the above, the affine scheme $\Spec(T^\pm(M))$ is a $\Gm$-torsor over $\Spec(R)$ in the Zariski topology,
where $\Gm$ denotes the usual affine group scheme
$$\Gm = \Spec(R[x,y]/(xy\sim 1)),$$
which is to say the group-valued functor $C\mapsto C^{\times}$.
\end{remark}

\begin{remark}
It follows that for every invertible module $M$, the unit and counit isomorphisms are inverses of one another.
This does not hold in all symmetric monoidal categories. 
(A counter-example is the category of super vector spaces, where
the odd one-dimensional vector space is invertible and self-dual but $\varepsilon\circ\eta=-1$.
We thank Kevin Coulembier and Kim Morrison for supplying this example.)
We do not know if this can be proved in our setting without proving some form of the results above.

A further consequence one can show is that $M_\pi$ is invertible if and and only if one has $\tr(\pi)=1$ and
all $2\times 2$ submatrices of $\pi$ are singular.
This criterion is evident over rings, 
and we find it satisfying that it remains true over semirings. 
It follows further that if $M_\pi$ is invertible, 
then the algebras $R[1/\pi_{ii}]$ cover $R$. So the trivializing cover $(R[1/\pi_{ij}])_{i,j}$ 
of~(\ref{thm:inv=loc-free-rank-1}) has a much smaller subcover.
\end{remark}

\subsection{Proof of (\ref{thm-intro: line bundles}) and table~\ref{table: implications}}
The right-pointing arrows in the table hold by the following, in left-to-right order:
\begin{enumerate}
	\item the Zariski topology is coarser than the fpqc topology,
	\item proposition \ref{pro: locally frees are dualizable},
	\item proposition \ref{proposition: dualizable iff f.g. proj.},
	\item proposition \ref{proposition: proj is flat}.
\end{enumerate}
The left-pointing arrows, in right-to-left order:
\begin{enumerate}
	\item proposition \ref{proposition: f.pres flats are projective},
	\item proposition \ref{proposition: dualizable iff f.g. proj.},
	\item standard commutative algebra,
	\item open question.
\end{enumerate}
The theorem~(\ref{thm-intro: vector bundles}) follows immediately.

\begin{definition}\label{def:Picard-group}
The \emph{Picard group}  $\textrm{Pic}(R)$ of $R$ is the set of isomorphism classes of invertible $R$-modules,
with group operation given by the tensor product $\otimes_R$.
\end{definition}

We note that there are no subtleties related to set-theoretic size here. 
The class $\mathrm{Pic}(R)$ is in bijection with a set
since invertible modules are finitely presented by~(\ref{proposition: dualizable iff f.g. proj.}).

\begin{remark}
The Picard group of schemes over semirings was first considered in~\cite{jun2017vcech}.
It was defined in terms of Zariski-locally free sheaves of rank 1. 	
(It was however in the different but equivalent language of locally semiringed topological spaces 
mentioned in~(\ref{rmk:Jaiung's thesis}).)

It follows from~(\ref{cor: invertible-iff-locally-rank-1}) that our definition above agrees with the
one in~\cite{jun2017vcech}.
\end{remark}

\begin{example}
$\textrm{Pic}(\mathbb{N})=\{0\}$. Let $L$ be an invertible $\mathbb{N}$-module. 
Since $L$ is invertible, it is flat, and hence the map $L\to \bZ\otimes_\bN L$ is injective, by 
(\ref{pro:flat-tensor-preserve-injectivity}).
And since $\textrm{Pic}(\mathbb{Z})=\{0\}$, we have $\bZ\otimes_\bN L \cong \bZ$. Composing these two maps,
we have an injection of $\mathbb{N}$-modules $e: L \to \mathbb{Z}$. 

Since $L$ is flat over $\mathbb{N}$, it is negative free, by (\ref{pro:neg-free-and-canc-flat-local}). It follows that $\textrm{im}(e) \subseteq \mathbb{N}$ or $\textrm{im}(e) \subseteq -\mathbb{N}$. Replacing $e$ by $-e$, we may assume that $\textrm{im}(e) \subseteq \mathbb{N}$. 

Further, we may also assume that $\mathrm{im}(e)$ generates the unit ideal in $\bZ$. 
Indeed, since $L \neq \{0\}$, the ideal $I$ generated by $\textrm{im}(e)$ is nonzero. Let $d$ be the unique generator of $I$ with $d>0$. Then the image of the map $\frac{1}{d}e:L \to \mathbb{Z}$ generates the unit ideal. 

Viewing $e$ as the inclusion map, we have a submodule $L\subseteq \bN$ which generates the unit ideal in $\bZ$.

Now we claim that any map $\chi:L \to \mathbb{N}$ extends uniquely across the inclusion $L \hookrightarrow \mathbb{N}$. Uniqueness is clear. So we consider existence.

First observe that for all $a,b\in L$, we have
\begin{equation}
	\label{eq:chi-1}
	a\chi(b) = \chi(ab) = b\chi(a).
\end{equation}
It follows that there exists a rational number $c\geq 0$ such that $\chi(a)=ca$ for all $a\in L$.
In fact, we have $c\in\bN$. Indeed, since $L$ generates the unit ideal in $\bZ$, we can choose $a$ and $b$ to be relatively prime. From (\ref{eq:chi-1}), we have $a\mid\chi(a)$ and hence $c\in \bN$. Therefore the map $\chi\:L\to\bN$ extends to a map $\bN\to\bN$, namely the one defined by same formula $a\mapsto ca$.

Thus we have an isomorphism
\[
\begin{tikzcd}
\Hom(\mathbb{N},\mathbb{N})=\mathbb{N}^\vee \arrow[r,"\simeq"] & \Hom(L,\mathbb{N})=L^\vee, 
\end{tikzcd}
\]
and hence the diagram
\[
\begin{tikzcd}
	L^{\vee\vee} \arrow[r,"\simeq"] &\mathbb{N}^{\vee\vee} \\
	L \arrow[r,hook] \arrow[u]& \mathbb{N} \arrow[u,"\simeq"]
\end{tikzcd}
\]
But, since $L$ is invertible, $L \to L^{\vee\vee}$ is an isomorphism. It follows that $L \to \mathbb{N}$ is an isomorphism. We conclude $\textrm{Pic}(\mathbb{N})=\{0\}$. 
\end{example}

\section{Reflexive modules}
\label{section: reflexive modules}

\begin{definition}
Let $R$ be a semiring. An $R$-module $M$ is said to be \emph{reflexive} if the canonical map $\delta_M:M \to M^{\vee\vee}$ is an isomorphism.
\end{definition}

\begin{remark}
We mention some facts which are easily proved but which we will not use.
Any dualizable module is reflexive. 
For finitely presented modules, reflexivity is an fpqc-local property.
The \emph{Dixmier projection} $(\delta_M)^\vee$ is a retraction of $\delta_{M^\vee}$:
$$
(\delta_M)^\vee \circ \delta_{M^\vee} = \id_{M^\vee}.
$$
\mpar{mention the double-dual monad?}
\end{remark}

If it happens that $M^{\vee\vee}$ is reflexive, as will be the case in the next section, 
then $M^{\vee\vee}$ is the reflexive completion of $M$:

\begin{proposition}
	\label{pro:reflexive-completion}
Let $f\:M\to P$ be a morphism of $R$-modules with $P$ reflexive. 
Then there exists a map $\tilde{f}\:M^{\vee\vee} \to P$ such that $f=\tilde{f}\circ \delta_M$. 
If $M^{\vee\vee}$ is reflexive, the map is unique.
\end{proposition}
\begin{proof}
Consider the following diagram:
\[
\begin{tikzcd}[column sep=1.8cm, row sep=1.3cm]
	M \arrow[r,"f"] \arrow[d,swap,"\delta_M"] 
		& P \arrow{d}{\delta_P}[swap]{\simeq}\\
	M^{\vee\vee} \arrow[r,"f^{\vee\vee}"] \arrow[ru,dashed, "g"] 
		 \arrow[bend right]{d}[swap]{(\delta_M)^{\vee\vee}} \arrow[bend left]{d}{\delta_{M^{\vee\vee}}}
		& P^{\vee\vee} \\
	M^{\vee\vee\vee\vee} \arrow[dashed]{ru}[swap]{g^{\vee\vee}}
\end{tikzcd}
\]
We are considering maps $g$ making the top-most triangle commute.

For existence, take $g=\delta_p^{-1}\circ f^{\vee\vee}$. This makes the second triangle commute: $\delta_P\circ g = f^{\vee\vee}$. Then  we have $f=g\circ \delta_M$ because $\delta_P$ is an isomorphism and the square commutes. 

For uniqueness, let $g$ now be any map such that $g\circ \delta_M=f$.
Therefore we have 
$$g^{\vee\vee}\circ (\delta_M)^{\vee\vee}=f^{\vee\vee}.$$ 
Now, we also have 
$$g^{\vee\vee}\circ\delta_{M^{\vee\vee}}=\delta_P\circ g.$$
Finally, because $M^{\vee\vee}$ is reflexive, we have
$$(\delta_M)^{\vee\vee}=\delta_{M^{\vee\vee}}.$$
Indeed,
$\delta_{M^{\vee\vee}}$ is an isomorphism, and hence its canonical retraction $(\delta_{M^\vee})^\vee$ is an
isomorphism, and hence its other section $(\delta_M)^{\vee\vee}$ agrees with $\delta_{M^{\vee\vee}}$.

Combining these equations, we have
$$
\delta_P\circ g = g^{\vee\vee}\circ \delta_{M^{\vee\vee}} = g^{\vee\vee}\circ (\delta_M)^{\vee\vee} = f^{\vee\vee},
$$
and hence $g=\delta_P^{-1}\circ f^{\vee\vee}$.
\end{proof}

\section{The narrow class group as a reflexive Picard group}
\label{section: narrow class group}

Let
\begin{align*}
	F &= \text{a number field},\\
	\Ded &= \text{a Dedekind domain the fraction field of which is }F,\\
	S &= \text{a set of real places of }F.
\end{align*}
By the term \emph{fractional ideal}, 
we mean an invertible (equivalently nonzero and finitely generated) sub-$\Ded$-module of $F$.
Fix the following notation:
\begin{align*}
	F_+ &=\{\alpha \in F \mid \forall \sigma \in S.\; \sigma(\alpha)\geq 0\}, \text{ which is a semifield}\\
	\Dedplus	&=	\Ded\cap F_+, \text{ which is a semiring} \\
	\operatorname{Id}(F) &= \text{the group of fractional ideals}\\
	\Cl_S(F) &= \operatorname{Id}(F)/\langle \beta\Ded \mid \beta\in F_+^{\times}\rangle
\end{align*}
Perhaps the most important case is where $\Ded$ is the ring of algebraic integers in $F$ 
and $S$ is the set of all real places.
Then ${\Cl_S}(F)$ is the classical narrow class group of $F$
(as defined in \cite{Narkiewicz:book}, p.\ 93, say).
In general, $\Cl_S(F)$ is a variant involving only the primes of $R$ and the real places in $S$.

We will say an $\Dedplus$-module $L$ is \emph{of rank 1} if  $\dim_F(F\otimes_{\Dedplus}L)=1$.
Then define
\begin{align*}
\Pic^{\mathrm{refl}}(\Dedplus) &= \{\text{isomorphism classes of
reflexive $\Dedplus$-modules of rank 1}\}.
\end{align*}
Again, there are no subtleties related to set-theoretic size since it will follow from the results below that
any such module can be embedded in $F$.

\begin{theorem}\label{thm:narrow-class-group} \ 
	\begin{enumerate}
		\item For any fractional ideal $I$ of $F$, 
			the $\Dedplus$-module $I_+:=I\cap F_+$ is reflexive and of rank 1.
		\item The map $$\operatorname{Id}(F)\to \Pic^{\mathrm{refl}}(\Dedplus)$$ defined by
			$I\mapsto I_+$ factors
			through the quotient map $\operatorname{Id}(F)\to{\Cl_S}(F)$.
		\item The resulting map is a bijection $${\Cl_S}(F)\longisomap \Pic^{\mathrm{refl}}(\Dedplus).$$
	\end{enumerate}
\end{theorem}

The proof will be given in (\ref{pf:narrow-class-group}), after we establish some preliminary facts.

\begin{lemma}\label{lemma: positive lemma 1}
Let $M$ be a sub-$\mathbb{N}$-module of an abelian group $M'$.
Then the induced map $i:\mathbb{Z}\otimes_\mathbb{N}M \to M'$ is an injection.  
\end{lemma}
\begin{proof}
Given any $x\in \ker(i)$, write $x=1\otimes m_1 + (-1)\otimes m_2$, where $m_1,m_2\in M$.
Since $i(x)=0$, we have $m_1-m_2=0$ and hence $m_1=m_2$. Therefore $x=0$.
\end{proof}

\begin{lemma}\label{lemma: positive lemma 2}
\begin{enumerate}
	\item The maps
$\mathbb{Z}\otimes_\mathbb{N}\Dedplus\to \Ded$ and $\mathbb{Q}\otimes_\mathbb{N}\Dedplus\to F$ are isomorphisms.
\item 
For a sub-$\Dedplus$-module $M$ of $F$, 
the map $\Ded\otimes_{\Dedplus} M\to F$ is injective, 
and $F\otimes_{\Dedplus} M \to F$ is bijective if $M$ is nonzero. 
\end{enumerate}
\end{lemma}
\begin{proof}
(1): By~(\ref{lemma: positive lemma 1}), the map $\mathbb{Z}\otimes_{\mathbb{N}}\Dedplus \to \Ded$, obtained from the inclusion $\Dedplus \hookrightarrow \Ded$, is injective. To show it is surjective, take $\alpha \in \Ded$. Then for a sufficiently large $n \in \mathbb{Z}$, we have $\sigma(n - \alpha) \geq 0$ for all $\sigma \in S$. Indeed, any $n \geq \sup\{\sigma(\alpha) \mid \sigma \in S\}$ works. Then, as $\alpha = n-(n-\alpha)$, our element $\alpha$ is the image of $1\otimes n +(-1) \otimes (n-\alpha) \in \mathbb{Z}\otimes_\mathbb{N}\Dedplus$. It follows that the map is surjective, and hence an isomorphism.

From this it follows that the second map $\mathbb{Q}\otimes_\mathbb{N}\Dedplus\to F$ is also an isomorphism.

(2): From (1) and~(\ref{lemma: positive lemma 1}), the composition
\[
\Ded\otimes_{\Dedplus}M = (\mathbb{Z}\otimes_\mathbb{N}\Dedplus)\otimes_{\Dedplus}M = \mathbb{Z}\otimes_\mathbb{N} M \hookrightarrow F
\]
is an injection.
Therefore so is the map $F\otimes_{\Dedplus}M \hookrightarrow F$.
It is surjective because the image is the $F$-linear span of $M$, which is assumed to be nonzero. 
\end{proof}

It follows that for nonzero sub-$\Dedplus$-modules $M$ and $N$ of $F$, the following composition
is injective:
\[
\Hom_{\Dedplus}(M,N) \hookrightarrow \Hom_{\Dedplus}(M,F) = \Hom_F(F\otimes_{\Dedplus}M, F) = \Hom_F(F,F)=F. 
\]
We can then identify $\Hom_{\Dedplus}(M,N)$ as follows:
\begin{equation}
	\label{eq:maps-in-terms-of-elements}
	\Hom_{\Dedplus}(M,N) =\{\beta \in F \mid \beta  M \subseteq N\}. 
\end{equation}
In particular, we obtain the following:
\begin{equation}\label{eq: identification}
M^\vee=\{\beta \in F \mid \beta   M  \subseteq \Dedplus\}. 
\end{equation}

\begin{proposition} \label{pro:reflexivity-in-term-of-elements}
	Let $M$ be a nonzero sub-$\Dedplus$-module of $F$. Then in terms of the above, we have the following:
	\begin{enumerate}
		\item The canonical map $\delta_M\:M\to M^{\vee\vee}$ is the inclusion $M\subseteq M^{\vee\vee}$.
		\item The inclusion $M\subseteq M^{\vee\vee}$ is an equality if and only if $M$ is reflexive.
		\item $M^\vee$ is reflexive.
		\item $M^{\vee\vee}$ is the reflexive completion of $M$.
	\end{enumerate}
\end{proposition}
\begin{proof}
(1): The canonical pairing $M^\vee\times M\to \Dedplus$ is given by $(\chi,\alpha)\mapsto \chi\alpha$.
Therefore $\delta_M$ is given by $\alpha\mapsto\alpha$.

(2): This follows from (1).

(3): Dualizing the containment $M\subseteq M^{\vee\vee}$, one obtains $M^\vee\supseteq M^{\vee\vee\vee}$.
Therefore $M^\vee$ is reflexive by (2).

(4): Apply~(\ref{pro:reflexive-completion}), noting that $M^{\vee\vee}$ is reflexive by (3).
\end{proof}

\begin{proposition} \label{pro:I-plus-is-reflexive}
Let $I$ be a fractional ideal of $F$, and let $I_+$ denote the sub-$\Dedplus$-module $I \cap F_+$. Then we have the following:
	\begin{enumerate}
		\item $\Ded\otimes_{\Dedplus}I_+ \to I$ is an isomorphism,
		\item $(I_+)^\vee=(I^\vee)_+$,
		\item $I_+$ is reflexive.
	\end{enumerate}
\end{proposition}
\begin{proof}
(1): Injectivity follows from~(\ref{lemma: positive lemma 1}) and~(\ref{lemma: positive lemma 2}):
\[
\Ded\otimes_{\Dedplus}I_+ = \mathbb{Z}\otimes_{\bN}{\Dedplus}\otimes_{\Dedplus} I_+ = \mathbb{Z}\otimes_\mathbb{N}I_+ \hookrightarrow I.
\]
To show surjectivity, let $\alpha_1,\dots,\alpha_n$ be generators for $I$. 
Take $m \in I\cap \mathbb{Z}$ such that $m \geq |\sigma(\alpha_i)|$ 
for all $\sigma \in S$ and $i \in \{1,\dots,n\}$. Let $\alpha_i'=\alpha_i+m$. 
Then $\alpha_1',\dots,\alpha_n'$ generate $I$ and for any $\sigma \in S$, 
we have $\sigma(\alpha_i')\geq 0$ for all $i$. 
It follows that $\alpha_i' \in F_+$, and hence $\alpha_i' \in I_+$. 
The map $\Ded\otimes_{\Dedplus}I_+ \to I$ is therefore surjective because its image is
a sub-$\Ded$-module containing the generators $\alpha_1',\dots,\alpha_n'$ of $I$. 
	
(2): First note that from \eqref{eq: identification}, we have
\begin{align*}
(I_+)^\vee &=\{\beta \in F \mid \beta  I_+ \subseteq \Dedplus\}\\
(I^\vee)_+ &=\{\beta \in F \mid \beta \in F_+ \textrm{ and } \beta  I \subseteq \Ded\}.
\end{align*}
Suppose $\beta \in (I_+)^\vee$. There exists a nonzero element $\alpha \in I_+$. 
Then $\beta\alpha \in \Dedplus$. 
It follows that for any $\sigma \in S$, we have $\sigma(\beta)\sigma(\alpha) \geq 0$. 
Since $\alpha \in I_+$ and $\alpha \neq 0$, we then have $\sigma(\beta)\geq 0$, which is to say $\beta \in F_+$. 
Also for any $\alpha_1,\alpha_2 \in I_+$, we have $\beta\alpha_1,\beta\alpha_2 \in \Dedplus$, 
and hence
	\[
	\beta(\alpha_1 - \alpha_2) = \beta\alpha_1 - \beta\alpha_2 \in \Ded. 
	\] 
It follows that $\beta  I \subseteq \Ded$. 
So $\beta \in I^\vee$, and hence $\beta \in I^\vee \cap F_+=(I^\vee)_+$. 

Now take $\beta \in (I^\vee)_+$. Then $\beta  I_+ \subseteq \beta  I \subseteq \Ded$. 
Also, $\beta \in F_+$. So $\beta  I_+ \subseteq F_+$, 
and hence $\beta   I_+ \subseteq F_+ \cap \Ded =\Dedplus$, 
showing that $\beta \in (I_+)^\vee$. 

(3): Applying (2) twice, we have 
	\[
	(I_+)^{\vee\vee} = ((I^\vee)_+)^\vee = (I^{\vee\vee})_+ = I_+.
	\]
Therefore $I_+$ is reflexive by~(\ref{pro:reflexivity-in-term-of-elements}), part (2).
\end{proof}

\begin{remark}
Is $I_+$ a flat $\Dedplus$-module? Is it finitely generated?
\end{remark}

\begin{lemma} \label{lem:reflexive-maps-injectively}
Let $L$ be a reflexive $\Dedplus$-module. Then $L \to F\otimes_{\Dedplus} L$ is injective.
\end{lemma}
\begin{proof}
It is enough to show that $L$ can be mapped injectively to some $F$-module, and this follows from reflexivity:
\[
L \longisomap L^{\vee\vee}=\Hom_{\Dedplus}(L^\vee,\Dedplus) \hookrightarrow \Hom_{\Dedplus}(L^\vee,F). 
\]
\end{proof}

\begin{proposition}\label{pro:surjectivity}
Let $L$ be a reflexive $\Dedplus$-module of rank 1.
\begin{enumerate}
	\item $L$ is isomorphic to a sub-$\Dedplus$-module $M$ of $F$ such that $M\cap F_+\neq \{0\}$.
	\item For any such $M$, let $I$ denote the sub-$\Ded$-module of $F$ generated by $M$. Then $I$
	is a fractional ideal and we have $M=I_+$.
\end{enumerate}
\end{proposition}
\begin{proof}
(1): Choose an isomorphism $F\otimes_{\Dedplus}L\to F$, and let $N$ denote the image of $L$.
By~(\ref{lem:reflexive-maps-injectively}), the map $L\to N$ is an isomorphism.
Now let $\beta \in N$ be a nonzero element, and let $M=\beta N$. Then $M$ is isomorphic to $N$ and hence to $L$, 
and $\beta^2$ is a nonzero element in $M\cap F_+$.

(2): Since $L$ is reflexive and nonzero, so is $M$. Therefore we have $M^\vee \neq \{0\}$. 
So there exists a nonzero $\chi \in F$ satisfying $\chi   M \subseteq \Dedplus$.
It follows that $M \subseteq \frac{1}{\chi}\Dedplus$ and hence 
$I\subseteq \frac{1}{\chi}\Ded$. 
Thus $I$ is a fractional ideal of $F$. 

Now consider the second statement. We first claim that $M \subseteq F_+$. 
By (1), there exists a nonzero $\alpha_0 \in M \cap F_+$ and, as above, a nonzero $\chi \in M^\vee$. 
Then since $\chi\alpha_0 \in \Dedplus\subseteq F_+$ and since $F_+$ is a semifield, 
we have $\chi\in F_+$.
Therefore $M^\vee\subseteq F_+$.
 
For every $\alpha \in M$, we have $\chi\alpha \in \Dedplus$. 
Again since $F_+$ is a semifield, we have $\alpha \in F_+$. It follows that $M\subseteq F_+$, and hence
$M\subseteq I \cap F_+=I_+$. 

To show $I_+\subseteq M$, consider any element $\alpha \in I_+$. 
Since $M$ is reflexive, it is enough to prove $\alpha \in M^{\vee\vee}$, 
or equivalently $\alpha  M^\vee \subseteq \Dedplus$. 
We will prove the following form: for any $\beta \in F$ such that 
$\beta   M \subseteq \Dedplus$, we have $\alpha\beta \in \Dedplus$. 

Since $\beta  M\subseteq \Dedplus$, 
we have $\beta   I \subseteq \Ded$ and hence $\alpha\beta \in R$. 
It remains to prove $\alpha\beta \in F_+$.
Since $\beta  M \subseteq F_+$ and $M \cap F_+ \neq \{0\}$, we have $\beta \in F_+$ and hence $\alpha \beta \in F_+$, since $\alpha \in F_+$.
\end{proof}

\begin{ssec}\label{pf:narrow-class-group}
\emph{Proof of~(\ref{thm:narrow-class-group}).}

(1): 
This follows from~(\ref{pro:I-plus-is-reflexive}).

(2):
Let $I$ be a fractional ideal in $F$ and $\beta \in F_+^{\times}$. 
Then we have an isomorphism of $\Dedplus$-modules $I_+ \to (\beta  I)_+$ sending $\alpha$ to $\beta\alpha$. 
It follows that the map $I\mapsto I_+$ factors through the narrow class group $\Cl_S(F)$ to define a map
$${\Cl_S}(F) \to \Pic^{\mathrm{refl}}(\Dedplus).$$

(3): The map is surjective by~(\ref{pro:surjectivity}). For injectivity, let $I$ and $J$ be fractional ideals
such that $I_+ \cong J_+$. By (\ref{eq:maps-in-terms-of-elements}), there exists a $\beta\in F$ such that $\beta
I_+ = J_+$. It follows that $\beta I=J$ and $\beta \in F_+^{\times}$, since $I_+ \neq \{0\}$. Therefore the
fractional ideals $I$ and $J$ agree in the narrow class group ${\Cl_S}(F)$.
\qed
\end{ssec}

\begin{remark}
	It follows that $\Pic^{\mathrm{refl}}(\Dedplus)$ classifies reflexive $\Dedplus$-modules
	which are invertible over $\sO_F$.

	A reflexive class group has also been considered in algebraic geometry over fields, 
	and especially for varieties $X$ with normal singularities,
	where it agrees with the Weil divisor class group. (See the Stacks Project~\cite{stacks-project}, tag 0EBK.)
	Whether there is some substance behind this analogy is not clear.
\end{remark}

\begin{remark}
It follows that the group structure on 	$\Pic^{\mathrm{refl}}(\Dedplus)$ is given by
	$$
	L\cdot L' = (L_\bZ\otimes_{\Ded}L'_\bZ)_+.
	$$
See also~\cite{stacks-project}, tag 0AVT.
\mpar{{\color{red}Is it also given by $(L\otimes_{\Dedplus}L')^{\vee\vee}$? Reference Stacks Project, tag 0EBK and  tag 0AVT, Reflexive Modules}}
\end{remark}

\begin{remark}
Take $R$ to be the ring $\mathcal{O}_F$ of algebraic integers in $F$, and $S$ to be the set of all real places.
Then $\Cl_S(F)$ is the usual narrow class group of $F$, 
which has a description as an adelic double quotient:
\begin{equation} \label{eq:narrow-adelic}
	\Cl_S(F) 
	= F^\times\backslash\sideset{}{'}\prod_{v} F_v^\times/\prod_{v}\mathcal{O}_{F_v}^\times	
\end{equation}
where $v$ runs over the non-complex places and we understand that at the real places, $\mathcal{O}_{F_v}$ denotes $\Rpos$.

The narrow class group appears as a certain compactly supported cohomology group 
$H^1_\mathrm{c}(\Spec(\mathcal{O}_F),\Gm)$ in Milne's book~\cite{Milne:arithmetic-duality-theorems}.
(See p.~207, Ch.\ II, remark 2.28.) 
It is a variant of the usual compactly supported cohomology
which incorporates the real places of $F$ in an \emph{ad hoc} manner using the Tate cohomology of $\Gal(\bC/\bR)$.
Milne writes, taking $X=\Spec(R)$:

\begin{quote}
Unfortunately $H^1_\mathrm{c}(X,\Gm)$ is not equal to the group of isometry classes of Hermitian invertible sheaves on $X$ (the ``compactified Picard group of $R$'' in the sense of Arakelov theory; see Szpiro\dots). I do not know if there is a reasonable definition of the \'etale cohomology groups of an Arakelov variety. Our definition of the cohomology groups with compact support has been chosen so as to lead to good duality theorems.
\end{quote}

Earlier, Deligne and Ribet~\cite{Deligne-Ribet}, (2.25), interpreted equation (\ref{eq:narrow-adelic}) in geometric but still \emph{ad hoc} terms. 
A narrow ideal class, and hence we imagine a
line bundle on the compactification $X\cup S$, is given by a trivialization over $F$ plus trivializations over each $\mathcal{O}_{F_v}$, modulo change of trivializations. 
This is standard, 
except that they understand a trivialization on a line at a real place to be a positive structure, 
or an orientation or a sign structure.

One could then view the present section as pointing out that by using semirings
we can implement the point of view of Deligne and Ribet systematically, 
thus offering a response to Milne's plaint.
\end{remark}

\begin{remark}
We can give a similar description of the reflexive Picard group for partially compactified arithmetic curves.
Let $T$ be a finite set of places of $F$ containing all the complex places.
Now consider $\overline{\Spec(\mathcal{O}_F)}-T$, 
the fully compactified arithmetic curve with the places in $T$ removed.
This can be written as $\Spec(R)\cup S$, where
$R$ is the subring of $F$ consisting of elements which are integral at all finite places not in $T$
and $S$ is the set of (real) infinite places not in $T$.

Then by (\ref{thm:narrow-class-group}) we have
	$$
	\Pic^{\mathrm{refl}}(\Dedplus) = \Cl_S(F)=
	F^\times\backslash\sideset{}{'}\prod_{v\not\in T} F_v^\times/\prod_{v\not\in T}\mathcal{O}_{F_v}^\times.
	$$
So up to change of trivializations, 
a reflexive $\Dedplus$-module of rank $1$ is given by a generic trivialization and
a formal trivialization at all $v\not\in T$, where again a trivialization over $\bR$ is a positive structure.
The formula above is the usual one for the Picard group of a curve missing the points in a finite set $T$, 
which is satisfying.
\end{remark}

\begin{example}
Let $F=\bQ(\sqrt{3})$, let $\Ded=\mathcal{O}_F$, 
and let $S$ consist of both the real places $\sigma_1, \sigma_2$. 
We will show that the ordinary Picard group $\Pic(\Dedplus)$ does not agree with the narrow class group
$\Pic^{\mathrm{refl}}(\Dedplus)$.
More specifically, we will show that the $\Dedplus$-module 
$$L:=\{\alpha\in \Ded\mid \sigma_1(\alpha)\geq 0, \sigma_2(\alpha)\leq 0\}$$ represents a class in
the reflexive Picard group but not the ordinary Picard group.

Observe that $L$ is of rank $1$, by (\ref{lemma: positive lemma 1}). It also satisfies $L^\vee=L$.
(One can show this by hand, or observe that $\sqrt{3}:L\isomap I_+$, 
where $I=(\sqrt{3})$, and apply (\ref{lemma: positive lemma 2}).)
In particular, it is reflexive.

But it is not invertible. To prove this, it is enough to show that the map 
	$$
	L\otimes_{\Dedplus} L=L\otimes_{\Dedplus} L^\vee\to \Dedplus
	$$ 
is not an isomorphism. 
So suppose for a contradiction that $1$ has a preimage $\sum_{i=1}^r \alpha_i\otimes \beta_i$. 
Thus we have $1=\sum_{i=1}^r \alpha_i\beta_i$.
We may assume $r\geq 1$, $\alpha_i\neq 0$, and $\beta_i\neq 0$. 
Since the $\alpha_i$ and $\beta_i$ have negative norm and the fundamental unit $2+\sqrt{3}$ has positive norm, 
they cannot be units. 
Therefore we have $N(\alpha_i),N(\beta_i)\leq -2$, and hence $N(\alpha_i\beta_i)\geq 4$.
So all the points $(\sigma_1(\alpha_i\beta_i),\sigma_2(\alpha_i\beta_i))\in\Rpos^2$ 
lie above the hyperbola $xy=4$.
But the point $(1,1)$ cannot be a sum of such points. So $1$ cannot be a sum of elements $\alpha_i\beta_i$.
\end{example}

\section{$\GL_n$ is not flat over $\bN$}
\label{section: GL_n is not flat}

This section and the next form an appendix to the paper. They are not used elsewhere. But we feel they
shed some light on how the theory of vector bundles differs when one passes from rings to semirings.
For instance over rings, moduli stacks are often constructed by quotienting by group schemes like $\GL_n$.
So the fact that it is rarely flat over semirings which are not rings would appear to present some complications.

For any $\bN$-algebra $R$, let $\GL_n(R)$ denote the group of automorphisms of the $R$-module $R^n$. 
We can identify it with the set of $n\times n$ matrices $M$ with a two-sided inverse.
(We note that by Reutenauer--Straubing~\cite{reutenauer1984inversion} every one-sided inverse of a square matrix 
over an arbitrary semiring is in fact a two-sided inverse, but we will have no need of this fact below.)
The functor $\GL_n$ is representable:
	$$
	\GL_n(R) = \Hom_{\bN}(A_n,R),
	$$
where $A_n$ is the semiring
	$$
	A_n = \bN[x_{ij}, y_{ij} \mid 1\leq i,j \leq n]/(\sum_k x_{ik}y_{kj}\sim \delta_{ij}, \sum_k y_{ik}x_{kj}\sim \delta_{ij}),
	$$
where $\delta_{ij}$ denotes the Kronecker delta function. 

Write $\GL_{n/R}$ for the restriction of $\GL_n$ to $\CAlg{R}$. So $\GL_{n/R} = \Spec(R\otimes_{\bN} A_n)$.

\begin{proposition}
	\label{pro:GL_n-not-flat}
	Suppose $n\geq 2$. Let $R$ be a negative-free semiring such that $\GL_{n/R}$ is flat over $R$.
	Then $\bZ\otimes_\bN R=\{0\}$
\end{proposition}
\begin{proof}
	For $i=1$ and $j=2$, we have the following relation in $A_n$ and hence in $R\otimes_{\bN} A_n$:
		$$
		\sum_k x_{1k}y_{k2} = 0.
		$$
	Since $R$ is negative free and $R\otimes_{\bN}A_n$ is flat over it, $R\otimes_{\bN}A_n$ is also negative 
	free, by (\ref{pro:neg-free-and-canc-flat-local}).
	Therefore we have $x_{1k}y_{k2}=0$ for all $k$. In particular $x_{12}y_{22}=0$.
	It follows that for any $R$-algebra $S$, every invertible matrix $(x_{ij})$ with entries in $S$ satisfies
	$x_{12}y_{22}=0$, where $(y_{ij})$ is the inverse of $(x_{ij})$.

	Now consider the $n\times n$ block-diagonal integer matrices
		$$
		(x_{ij})=\left(
		\begin{array}{rr|r}
			1 & -1 \\ -1 & 2 \\ \hline & & 1_{n-2}
			\end{array}
		\right), \quad
		(y_{ij})=\left(
		\begin{array}{cc|c}
			2 & 1 \\ 1 & 1 \\ \hline & & 1_{n-2}
			\end{array}
		\right),
		$$
	which are inverses of one another. Their images in $\GL_n(\bZ\tn_\bN R)$
	are also inverses, and hence
	by the above, satisfy $x_{12}y_{22}=0$ in $\bZ\tn_\bN R$. 
	But $x_{12}y_{22}=-1$. Therefore we have $-1=0$ in $\bZ\tn_\bN R$, and hence 
	$\bZ\otimes_\bN R$ is the 
	zero ring. 
	\end{proof}

\begin{corollary}
\label{cor: GL_n is not flat}
If $n\geq 2$, then $\GL_n$ is not flat over $\bN$, $\Qpos$, or $\Rpos$, or indeed any nonzero negative-free 
additively cancellative semiring.
\end{corollary}
\begin{proof}
Let $R$ be such a semiring. If $\GL_{n/R}$ were flat, by~(\ref{pro:GL_n-not-flat}) the map
$$
R \longmap \bZ\otimes_\bN R = \{0\}
$$
would be injective, since $R$ is additively cancellative. This contracts the assumption $R\neq\{0\}$.
\end{proof}

\section{$\GL_n$ near the Boolean fiber}
\label{section: GL_n near the Boolean fiber}

A commutative monoid $M$ satisfies $\bB\otimes_\bN M=\{0\}$ if and only if $M$ is an abelian group. 
(See~(7.9) of \cite{borger2016boolean}.) 
So one might say in more scheme-theoretic language that an $\bN$-module is zero over the Boolean point of
$\Spec(\bN)$ if and only if it lies over the subscheme $\Spec(\bZ)$.
One might therefore say that an $\bN$-module $M$ is zero \emph{away} from the Boolean point
if $\bZ\otimes_\bN M$ is zero. 

The purpose of this section is then to show that over a semiring which is zero away from the Boolean point,
the group scheme $\GL_n$ reduces to the torus normalizer $\Gm^n\rtimes S_n$. 

First, we give some equivalent conditions:

\begin{proposition}\label{pro:universally negative-free equiv conditions}
	Let $R$ be a semiring. Then the following are equivalent:
	\begin{enumerate}
		\item every $R$-algebra is negative free (``$R$ is universally negative free''),
		\item $\bZ\otimes_\bN R = \{0\}$ (``$R$ is zero away from the Boolean point''),
		\item there exists an element $r\in R$ such that $r+1=r$.
	\end{enumerate}
\end{proposition}
\begin{proof}
	(1)$\Rightarrow$(2): Since $\bZ\otimes_\bN R$ is both an $R$-algebra and a $\bZ$-algebra, 
	it is both negative free and an abelian group.
	It therefore must be zero.

	(2)$\Leftrightarrow$(3): $\bZ\otimes_\bN R$ is isomorphic (canonically) to the group completion of $R$, 
	which is the set of	pairs $(a,b)$ (thought of as the formal difference $a-b$) modulo the relation 
		$$
		[(a,b)\sim (a',b')] = [\exists r.\; a+b'+r=a'+b+r].
		$$
	Since $R$ is a semiring (and not just a commutative monoid), 
	the group completion $\bZ\otimes_\bN R$ is a ring.
	It is therefore zero if and only if $1=0$ holds in it, 
	which is to say the relation $(1,0)\sim (0,0)$ holds.
	This is equivalent to the existence of an $r\in R$ such that $1+r=r$.
	
	(3)$\Rightarrow$(1):
	Suppose $x+y=0$ for some $x$ and $y$ in an $R$-algebra.
	Then we have
	\begin{align*}
	x + r x + r y &= x + r (x+y) = x,\\
	x + r x + r y &= (r+1)x +r y = r x + r y = r (x+y) = 0.	
	\end{align*}
	Combining these, we conclude $x = 0$, and hence the algebra is negative free.
\end{proof}

\begin{example}
The following semirings, and all algebras over them, are therefore universally negative free:
$\bB$, $\bN/(r+1\sim r)$, 
	$$
	\bN\cup\{\infty\}=\lim_{r\geq 0} \bN/(r+1\sim r),
	$$
and $\bN[r]/(r+1\sim r)$,
the last of which is weakly universal by the proposition above.
\end{example}

\begin{remark}
	We note that while being universally negative free and being zero away from the Boolean point
	appear to be formulated as genuine properties, the existence of an element $r$ such that $r+1=r$
	is formulated as the property underlying a structure---namely the data of such an element $r$.
	And indeed, the element $r$ is not generally unique. 
	For instance, in $\bB[x]$, any polynomial of the form $1+xf(x)$ satisfies $r+1=r$.
	Equivalently, the map $\bN\to\bN[r]/(r+1\sim r)$ is not an epimorphism of semirings, unlike
	for instance $\bN\to\bZ$. 
\end{remark}

\begin{theorem} 
	\label{thm:GLn-over-univ-neg-free}
	Let $R$ be a semiring satisfying $\bZ\otimes_\bN R=\{0\}$.
	Then the canonical map of group schemes
		$$
		{\mathbold G}_{\mathrm{m}/R}^n\rtimes \underline{S_n} \to \GL_{n/R}
		$$ 
	is an isomorphism. 
	Here, $\underline{S_n}$ denotes $\Spec(\prod_{S_n} R)$, the $n$-th symmetric group viewed as a constant group scheme 
	over $R$.
\end{theorem} 
\begin{proof}
	By~(\ref{pro:universally negative-free equiv conditions}),
	every $R$-algebra is negative free.
	Therefore the defining equations for $\GL_{n/R}$ can be simplified to the following:
		$$
		\sum_k x_{ik}y_{ki}=1,\quad x_{ik}y_{kj}= 0,\quad 
		\sum_k y_{ik}x_{ki}= 1,\quad y_{ik}x_{kj}= 0\quad (j\neq i).
		$$
	(We will actually ignore the last two equations, 
	which might be unsurprising in light of the fact that they are redundant~\cite{reutenauer1984inversion}.)
	To check that the monomorphism 
	${\mathbold G}_{\mathrm{m}/R}^n\rtimes \underline{S_n} \to \GL_{n/R}$ is an isomorphism, 
	it is enough to show that it has a section Zariski-locally on $R$. 
	Therefore by the first equation above, we may assume that for all $i$, 
	there exists a $k_i$ such that $x_{ik_i}y_{k_i i}$ is a unit, and hence that $x_{ik_i}$ is a unit. 
	We then have $y_{k_i j}=0$ for all $j\neq i$, by the second equation.
	Thus for all $i$, the $k_i$-th row of the matrix $(y_{ij})$ 
	is zero except for the entry in the $i$-th column, which is a unit.
	It follows that the function $i\mapsto k_i$ is injective (we may assume $R\neq \{0\}$, since otherwise
	the theorem is trivially true) and hence a permutation.
	Therefore the matrix $(y_{ij})\in\GL_n(R)$ is in the image of $\Gm(R)\rtimes S_n$.
\end{proof}

\begin{remark}
	The fact $\GL_n(R)=(R^{\times})^n \rtimes S_n$ for semirings $R$ which are negative free and for which 
	$\Spec(R)$ is connected (i.e., $R$ is nonzero and has no nontrivial idempotent pairs) has already appeared
	in~\cite{jun2024vector}, corollary 3.19 and proposition 4.6. 
	The theorem above is not much more than a scheme-theoretic expression of this fact.
\end{remark}

\begin{corollary}
	$\GL_{n/R}$ is flat over any semiring $R$ satisfying $\bZ\otimes_\bN R=\{0\}$.
\end{corollary}
\begin{proof}
	Indeed, by (\ref{thm:GLn-over-univ-neg-free}), we have $\sO(\GL_{n/R})=\prod_{S_n}R[x^{\pm1}]$, which
	is free and hence flat over $R$.
\end{proof}

\begin{remark}
It follows that~(\ref{pro:GL_n-not-flat}) has a converse of a kind:
A semiring is universally negative free if and only if it is negative free and $\GL_n$ is flat over it
for all $n$.	
\end{remark}

\bibliography{references}
\bibliographystyle{amsalpha}

\end{document}